\documentclass{amsart}
\usepackage{enumerate}
\usepackage{amssymb}
\usepackage{color}

\newtheorem{theorem}{Theorem}[section]
\newtheorem{corollary}[theorem]{Corollary}
\newtheorem{definition}[theorem]{Definition}
\newtheorem{lemma}[theorem]{Lemma}
\newtheorem{proposition}[theorem]{Proposition}

\newtheorem{remark}[theorem]{Remark}
\newtheorem{example}[theorem]{Example}

\newcommand{\M}{{\mathcal M}}
\newcommand{\N}{{\mathcal N}}

\numberwithin{equation}{section}

\begin{document}

\title[\emph{}]{Johnson-Schechtman inequalities for noncommutative martingales}

\author[]{Yong Jiao}
\address{School of Mathematics and Statistics, Central South University, Changsha 410075, China}
\email{jiaoyong@csu.edu.cn}
\curraddr{School of Mathematics and Statistics, University of NSW, Sydney,  2052, Australia}
\email{yong.jiao@unsw.edu.au}

\author[]{Fedor Sukochev}
\address{School of Mathematics and Statistics, University of NSW, Sydney,  2052, Australia}
\email{f.sukochev@unsw.edu.au}

\author[]{Dmitriy  Zanin}
\address{School of Mathematics and Statistics, University of NSW, Sydney,  2052, Australia}
\email{d.zanin@unsw.edu.au}

\author[]{Dejian Zhou}
\address{School of Mathematics and Statistics, Central South University, Changsha 410075, China}
\email{zhoudejian@csu.edu.cn}



\thanks{ Yong Jiao is supported by NSFC(11471337) and Hunan Province NSF(14JJ1004).}

\begin{abstract} In this paper we study Johnson-Schechtman inequalities for noncommutative martingales. More precisely, disjointification inequalities of noncommutative martingale difference sequences are proved in an arbitrary symmetric operator space $E(\mathcal{M})$ of a finite von Neumann algebra $\mathcal{M}$ without making any assumption on the Boyd indices of $E$. We show that we can obtain Johnson-Schechtman inequalities for arbitrary martingale difference sequences and that, in contrast with the classical case of independent random variables or the noncommutative case of freely independent random variables, the inequalities are one-sided except when $E=L_2(0,1)$. As an application, we partly resolve a problem stated by Randrianantoanina and Wu in \cite{NW}. We also show that we can obtain sharp $\Phi$-moment analogues for Orlicz functions satisfying $p$-convexity and $q$-concavity for $1 \leq p \leq 2$, $q=2$ and $p=2$, $2< q < \infty$. This is new even for the classical case. We also extend and strengthen the noncommutative Burkholder-Gundy inequalities in symmetric spaces and in the $\Phi$-moment case.
\end{abstract}

\maketitle

\section{Introduction}
Let $(f_k)_{k=0}^n\subset L_p(0,1), 2<p<\infty,$ be a sequence of independent mean zero random variables. In 1970, Rosenthal \cite{R} proved that
$$\Big\|\sum_{k=1}^nf_k\Big\|_p\approx_{p} \Big \|\bigoplus _{k=0}^nf_k\Big\|_{L_p\cap L_2},\quad 2<p<\infty,$$
where $\bigoplus _{k=0}^nf_k:=\sum_{k=1}^nf_k(\cdot-k+1)\chi_{[k-1,k)}$ is a disjoint sum of random variables $(f_k)_{k=0}^n$ which is a Lebesgue measurable function on $(0,\infty)$.
Carothers and Dilworth extended the results above to the case of Lorentz spaces in \cite{CD1}.
In the setting of symmetric function spaces, Johnson and Schechtman \cite{JS} established a far reaching generalisation of the Rosenthal inequality. If $E$ is a symmetric space on $[0,1]$ (see e.g. \cite{KPS,LT}) and if $1\leq p\leq \infty,$ then the set $Z_E^p$ consists of all measurable functions on $(0,\infty)$ for which
$$\|f\|_{Z_E^p}:=\|\mu(f)\chi_{[0,1]}\|_E+\|f\|_{L_1+L_p}<\infty.$$
Here, $\mu(f)$ denotes a decreasing rearrangement of the function $f$ (see \cite{KPS,LT} where the symbol $f^*$ is used and   Section 2 below).
Johnson and Schechtman \cite{JS} proved that
\begin{equation}\label{JS}
\Big\|\sum_{k=0}^nf_k\Big\|_E\approx_E \Big \|\bigoplus _{k=0}^nf_k\Big\|_{Z_E^2},
\end{equation}
for every sequence $(f_k)_{k=0}^n$ of independent mean zero random variables on $(0,1)$ whenever $L_p\subset E$ for some $p<\infty.$

Above, and in what follows, we write $A\lesssim_E B$ if there is a constant $C_E>0$ depending only on $E$ such that $A\leq C_E B$; and we write $A\approx_E B$ if both $A\leq C_E B$ and $B\leq C_E A$ hold, possibly with different constants; we write $A\approx B$ if the inequalities above hold for an absolute constant $C$ which is independent of $E.$

Astashkin and Sukochev introduced the concept of a Kruglov operator $K$ and extended Johnson-Schechtman inequalities in \cite{AS2005} (see also \cite{pacific}). Namely, it is proved that \eqref{JS} holds if and only if the Kruglov operator $K$ is bounded on a symmetric space $E$ (with Fatou norm). The latter condition is far less restrictive than the assumption that $L_p \subset E$ for some $p<\infty$ (see \cite[Section 7]{AS2005}). Recently, the operator approach of \cite{AS2005,pacific} was extended into the realm of (noncommutative) free probability theory in \cite{SZ}. By using a free Kruglov operator, a version of Johnson-Schechtman inequalities in the setting of free probability theory was obtained. We briefly recall their main results. Let $\M$ be a finite von Neumann algebra equipped with a faithful normal normalised trace. Let $E(\M)$ be a symmetric Banach operator space (\cite{KS}) equipped with a Fatou norm. Then\footnote{Here, $(e_k)$ denotes the standard basic sequence of $\ell_\infty$. In the case of $M=L_\infty(0,1)$, the distribution of $\sum_{k\geq0}f_k\otimes e_k$ and that of $\bigoplus_{k\geq0}f_k$ coincide.}
\begin{equation}\label{SZ}
\Big\|\sum_{k=0}^nx_k\Big\|_{E(\M)}\approx \Big \|\sum _{k=0}^nx_k\otimes e_k\Big\|_{Z_E^2(\M\bar{\otimes} \ell_\infty)}
\end{equation}
for every sequence $(x_k)$ of freely independent symmetrically distributed random variables from $E(\M).$ In the special case $E=L_p$ for $1\leq p\leq \infty$, this result was proved by Junge, Parcet and Xu \cite[Theorem A]{JPX}. For $p=\infty$, it is due to Voiculescu \cite{V1}.

In this paper, we replace the sequence of independent (or freely independent) random variables with a radically more general case: an arbitrary sequence of noncommutative martingale differences. The notation
$Int(L_p(0,1),L_q(0,1))$ stands for the set of all interpolation spaces for the couple $(L_p(0,1),L_q(0,1))$, $0<p<q\leq\infty.$ Our first main result can be stated as follows.

\begin{theorem}\label{first main} Let $E$ be a symmetric quasi-Banach space on $(0,1),$ let $\mathcal{M}$ be a finite von Neumann algebra, and let $(x_k)_{k\geq0}\subset E(\M)$ be a sequence of martingale differences.
\begin{enumerate}[{\rm (i)}]
\item\label{fma} If $E\in Int(L_1(0,1),L_2(0,1)),$ then
\begin{equation}\|\sum_{k\geq0}x_k\|_{E(\M)}\lesssim_E\|\sum_{k\geq0}x_k\otimes e_k\|_{Z_E^2(\M\bar{\otimes}\ell_\infty)}.\end{equation}
\item\label{fmb} If $E\in Int(L_2(0,1),L_{\infty}(0,1)),$ then
\begin{equation}\label{norm 2}\|\sum_{k\geq0}x_k\otimes e_k\|_{Z_E^2(\M\bar{\otimes} \ell_\infty)}\lesssim_E\|\sum_{k\geq0}x_k\|_{E(\M)}.\end{equation}
\end{enumerate}
\end{theorem}

Even in a very special case, when $\M$ is a commutative von Neumann algebra $L_\infty(0,1)$ and $E$ is a symmetric Banach space, Theorem \ref{first main} strengthens Theorem 3.1 and Theorem 3.5 (or Corollary 3.2 and Corollary 3.6) from \cite{ASW}. The different method used in this present paper allows us to remove some extra assumptions (e.g., separability, order semi-continuity of the norm and the upper Boyd index $p_E<\infty$) imposed on the symmetric space $E$ in \cite{ASW}. Moreover, we consider the  harder case of quasi-Banach spaces in Theorem \ref{first main}, which renders the duality arguments used in \cite{ASW} inapplicable.

In sharp contrast with the setting \eqref{JS} (or \eqref{SZ}) of independent (or freely independent) random variables, only one-sided Johnson-Schechtman inequalities hold for martingale differences. More precisely, a two-sided disjointification inequality for noncommutative martingale differences holds only in $L_2.$

\begin{theorem}\label{js 2 sides} Let $E$ be a separable symmetric Banach space on $(0,1).$ If for every finite von Neumann algebra $\mathcal{M}$ and an arbitrary sequence $(x_k)_{k\geq0}\subset E(\M)$ of martingale differences we have
\begin{equation}\label{js equality}
\|\sum_{k\geq0}x_k\|_{E(\M)}\approx_{E}\|\sum_{k\geq0}x_k\otimes e_k\|_{Z_E^2(\M\bar{\otimes} \ell_\infty)},
\end{equation}
then $E=L_2(0,1).$
\end{theorem}

{\color{red}Again}, in the special case of $\M=L_\infty(0,1)$, Theorem \ref{js 2 sides} strengthens the corresponding result of \cite[Corollary 3.8]{ASW}. The assumptions of Theorem \ref{js 2 sides} is much weaker than those in \cite{ASW}.

The following theorem extends \cite[Theorem 2.1]{PX}, and strengthens \cite[Proposition 4.18]{DPPS} and \cite[Theorem 3.1]{J-BG}. Again, using our methods, extra assumptions on the space $E$ imposed in \cite{DPPS} (e.g. $2$-convexity) are eliminated.

\begin{theorem}\label{second main} Let $E$ be a symmetric Banach space on $(0,1)$, let $\mathcal{M}$ be a finite von Neumann algebra, and let $(x_k)_{k\geq0}\subset E(\M)$ be a sequence of martingale differences.
\begin{enumerate}[{\rm (i)}]
\item\label{sma} If $E\in Int(L_p(0,1),L_2(0,1))$ for some $1<p<2,$ then
\begin{equation}\label{bgundy leq2}
\|\sum_{k\geq0}x_k\|_{E(\M)}\approx_{E}\inf_{x_k=y_k+z_k}\Big(\Big\|\Big(\sum_{k\geq0}|y_k|^2\Big)^{1/2}\Big\|_{E(\M)}+\Big\|\Big(\sum_{k\geq0}|z_k^*|^2\Big)^{1/2}\Big\|_{E(\M)}\Big),
\end{equation}
where the infimum is taken over the sequences $(y_k) $ and $(z_k)$ of martingale differences.
\item\label{smb} If $E\in Int(L_2(0,1),L_q(0,1))$ for some $2<q<\infty,$ then
\begin{equation}\label{bgundy geq2}
\|\sum_{k\geq0}x_k\|_{E(\M)}\approx_{E}\max\Big\{\Big\|\Big(\sum_{k\geq0}|x_k|^2\Big)^{1/2}\Big\|_{E(\M)},\Big\|\Big(\sum_{k\geq0}|x_k^*|^2\Big)^{1/2}\Big\|_{E(\M)}\Big\}.
\end{equation}
\end{enumerate}
\end{theorem}

Theorem \ref{second main} is sharp in the following sense.

\begin{theorem}\label{second l2} Let $E$ be a symmetric Banach space on $(0,1).$ Suppose that for every sequence $(x_k)_{k\geq0}\subset E(\M)$ of martingale differences the equality \eqref{bgundy leq2} (or \eqref{bgundy geq2}) holds. It follows that $E\in Int(L_p,L_q)$ for $1<p<q<\infty.$
\end{theorem}

Theorem \ref{first main} and Theorem \ref{js 2 sides} (respectively, Theorem \ref{second main} and Theorem \ref{second l2}) are proved in Section 3 (respectively, Section 4). Section \ref{burk section} is devoted to the noncommutative Burkholder inequality. In the setting of $L_p$-spaces, it  was proved by Junge and Xu \cite{JX1} and was later extended to Lorentz spaces in \cite{J-B} by using the weak type $(1,1)$ decomposition of Randrianantoanina \cite{N2}. Recently, it was established in \cite{dirksen2015} (respectively, in \cite{NW}) for a symmetric operator space $E(\mathcal{M})$ such that $E\in Int(L_p,L_q)$ for $2<p<q<\infty$ (respectively,  $E\in Int(L_p,L_q)$ for $1<p<q<2$). However, the Burkholder inequality for the case when $E\in Int(L_2,L_q)$ for some $2<q<\infty$ (or $E\in Int(L_p,L_2)$ for some $1<p<2$ ) was stated in \cite{NW} as an open question. Applying Theorem \ref{first main} and Theorem \ref{second main}, we partly resolve this problem. Our second main result is stated as follows; see Subsection 2.6 for the unexplained notation.

\begin{theorem}\label{the 2-4} Let $E$ be a symmetric Banach space on $(0,1)$, let $\mathcal{M}$ be a finite von Neumann algebra, and let $(x_k)_{k\geq0}\subset E(\M)$ be a sequence of martingale differences. If $E\in Int(L_2(0,1),L_4(0,1)),$ then
\begin{eqnarray*}\|\sum_{k\geq0}x_k\|_{E(\M)}&\approx_{E}&\|\sum_{k\geq0}x_k\otimes e_k\|_{Z^2_E(\M\bar{\otimes}\ell_\infty)}
\\&+&\Big\|\Big(\sum_{k\geq0}\mathcal{E}_{k-1}(|x_k|^2)\Big)^{1/2}\Big\|_{E(\M)}+\Big\|\Big(\sum_{k\geq0}\mathcal{E}_{k-1}(|x_k^*|^2)\Big)^{1/2}\Big\|_{E(\M)}.
\end{eqnarray*}
\end{theorem}

Unfortunately, at the time of writing this paper, we still do not know how to extend this result to the case $E\in Int(L_2,L_q)$ for some $4<q<\infty.$ The partial achievement in this direction is Theorem \ref{burk conditional}.

Our second aim in this paper is to prove some $\Phi$-modular analogues of Johnson-Schechtman inequalities for noncommutative martingales. Let $\Phi$ be an Orlicz function on $(0,\infty)$. The study of $\Phi$-moment inequalities for martingales was initiated by Burkholder and Gundy in their
remarkable paper \cite{BG}. Since then, most of the classical $p$-th moment inequalities for martingales were transferred to $\Phi$-moment inequalities; see \cite{BDG,G}. In recent years, $\Phi$-moment inequalities have been extended to noncommutative martingales. We refer to \cite{BC,DR} for $\Phi$-moment versions of noncommutative Burkholder-Gundy inequalities and also refer to \cite{BCO} for noncommutative maximal inequalities for convex functions.
In very recent papers \cite{JSXZ,JSZ-LMS}, the substantial difference between \eqref{JS} as well as \eqref{SZ} and their $\Phi$-moment analogues has been demonstrated in the case of independent and noncommutative independent random variables. Motivated by the results above, it is natural to consider $\Phi$-moment analogue of Theorem \ref{first main}. The following is our third main result.

\begin{theorem}\label{third main} Let $\M$ be a finite von Neumann algebra and $\Phi$ be an Orlicz function, and let $\{x_k\}_{k\geq0}\subset L_\Phi(\M)$ be a sequence of martingale differences.
\begin{enumerate}[{\rm (i)}]
\item\label{tma} If $\Phi$ is $2$-concave, then
\begin{equation}
\tau\Big(\Phi\Big(|\sum_{k\geq0}x_k|\Big)\Big)\lesssim{\color{red}\mathbb E}(\Phi(\mu(X)\chi_{(0,1)}))
+\Phi(\|X\|_{L_1+L_2}),\quad X=\sum_{k\geq0}x_k\otimes e_k.
\end{equation}
\item\label{tmb} If $\Phi$ is $2$-convex and $q$-concave for some $2<q<\infty,$ then
\begin{equation}\label{modular 2}
{\color{red}\mathbb E}(\Phi(\mu(X)\chi_{(0,1)}))+\Phi(\|X\|_{L_2})\lesssim_{\Phi}\tau\Big(\Phi\Big(|\sum_{k\geq0}x_k|\Big)\Big),\quad X=\sum_{k\geq0}x_k\otimes e_k.
\end{equation}
\end{enumerate}
\end{theorem}

The assumptions on the Orlicz function $\Phi$ in Theorem \ref{third main} \eqref{tmb} are sharp (see Proposition \ref{q-concavity necessary}). This allows us to make some interesting comparisons between the modular inequality \eqref{modular 2} and the Orlicz norm inequality \eqref{norm 2}. Let $\Phi$ be a $2$-convex Orlicz function which fails to be $q$-concave for all finite $q.$ If $L_\Phi(0,1)$ is the Orlicz space\footnote{see Definition in Section \ref{prelims section}} associated with $\Phi,$ then $L_\Phi(0,1)\in Int(L_2(0,1),L_\infty(0,1))$ (see e.g. \cite{LSh}). Hence,  by Theorem \ref{first main} (\ref{fmb}) the inequality
$$\|\sum_{k\geq0}x_k\otimes e_k\|_{Z_{L_\Phi}^2(\M\bar{\otimes} \ell_\infty)}\lesssim_\Phi\|\sum_{k\geq0}x_k\|_{L_\Phi(\M)}$$
holds for an arbitrary sequence $(x_n)_{n\geq0}\subset L_\Phi(\M)$ of martingale differences, while the corresponding $\Phi$-moment inequality \eqref{modular 2} fails due to Proposition \ref{q-concavity necessary}. This indicates an important difference between (noncommutative and classical) Johnson-Schechtman inequalities for martingales with respect to  symmetric norms and their $\Phi$-moment analogues.

Finally, in Section 7, we obtain $\Phi$-moment Burkholder-Gundy inequalities for a $p$-convex and $q$-concave Orlicz function $\Phi.$ Our result in this section extends the main result in \cite[Theorem 5.1]{BC}. We also give some examples demonstrating that our result can be applied to a larger class of Orlicz functions than that in \cite{BC}.

\section{Preliminaries}\label{prelims section}

\subsection{Singular value function}
Let $\mathcal{M}$ be a finite (or semifinite) von Neumann algebra equipped with a faithful normal finite (or semifinite) trace $\tau$. We denote by $L_0(\mathcal{M},\tau)$, or simply $L_0(\mathcal{M})$, the family of all $\tau$-measurable operators \cite{FK}.
Recall that $e_{(s,\infty)}(|x|)$ is the spectral projection of $x\in L_0(\mathcal{M})$ associated with the interval $(s,\infty)$. For $x\in L_0(\mathcal{M})$, the generalized singular number is defined by
$$\mu(t,x)=\inf\{s>0: \tau(e_{(s,\infty)}(|x|))\leq t\}, \quad t>0.$$
The function $t\mapsto\mu(t,x)$ is decreasing and right-continuous \cite{FK}. For the case that $\M$ is the abelian von Neumann algebra $L_\infty(0,1)$ with the trace given by integration with respect to the Lebesgue measure,  $L_0(\mathcal{M})$ is the space of all measurable functions and $\mu(f)$ is the decreasing rearrangement of the measurable function $f$; see \cite{KPS,LT}.

\subsection{Symmetric operator spaces} A quasi-Banach function space $(E,\|\cdot\|_E)$ on the interval $(0,\alpha)$, $0<\alpha\leq \infty$ is called symmetric if for any $g\in E$ and for any measurable function $f$ with $\mu(f)\leq\mu(g)$, we have $f\in E$ and $\|f\|_E\leq\|g\|_E.$

For a given symmetric quasi-Banach space $(E,\|\cdot\|_E)$, we define the corresponding noncommutative space on $(\M,\tau)$ by setting
$$E(\M,\tau):=\{x\in L_0(\M,\tau):\mu(x)\in E\}.$$
Endowed with the quasi-norm $\|x\|_{E(\M)}:=\|\mu(x)\|_E,$ the space $E(\M,\tau)$ is called the noncommutative symmetric space associated with $(\M,\tau)$ corresponding to the function space $(E,\|\cdot\|_E)$. It is shown in \cite{S} that the quasi-norm space $(E(\M),\|\cdot\|_{E(\M)})$ is complete if $(E,\|\cdot\|_E)$ is complete. For every semifinite von Neumann algebra $\N$, the space $Z_E^2(\N)$ is well defined due to \cite{KS}.

The following useful construction can be found in \cite{JS}. If $E$ is a symmetric quasi-Banach space on the interval $(0,1)$, then the space $Z_E^2$ consists of all measurable functions $f\in L_1(0,\infty)+L_\infty(0,\infty)$ for which
$$\|f\|_{Z_E^2}:=\|\mu(f)\chi_{(0,1)}\|_E+\|f\|_{L_1+L_2}<\infty.$$
It is obvious that the functional $\|\cdot\|_{Z_E^2}$ is a quasi-norm on $Z_E^2.$
For $f\in L_1(0,\infty)+L_\infty(0,\infty)$ and $t>0,$ it is well known that
\begin{multline}\label{1+2}\frac14\Big(\int_0^1\mu(t,f)dt+(\int_1^\infty\mu^2(t,f)dt)^{1/2}\Big)
\\\leq \|f\|_{L_1+L_2}\leq\int_0^1\mu(t,f)dt+\Big(\int_1^\infty\mu^2(t,f)dt\Big)^{1/2}.
\end{multline}
Thus it is easy to see that
$$\|f\|_{Z_E^2}\approx\|\mu(f)\chi_{(0,1)}\|_E+\|f\chi_{[1,\infty)}\|_{L_2}.$$
In what follows, our other unexplained terminologies concerning symmetric function spaces are standard; see for example \cite{LT} for definitions of Boyd indices, $p$-convexity, etc.

\subsection{Interpolation} Let $F_1,F_2$ and $F_1{\color{red}\cap }F_2\subset E\subset F_1+F_2$ be symmetric quasi-Banach spaces. We say that $E$ is an interpolation space between $F_1$ and $F_2$ (written $E\in Int(F_1,F_2)$) if for every $T:F_1+F_2\to F_1+F_2$ such that $T:F_1\to F_1$ and $T:F_2\to F_2$ is bounded, we also have $T:E\to E$ is bounded. We refer to \cite{KPS} and \cite{KM} for some necessary background on interpolation.
For $1\leq r<\infty,$ let $E_{(r)}$ be $r$-concavfication of $E$, that is, $\|f\|_{E_{(r)}}=\||f|^{1/r}\|_E^r.$ It is proved \cite[Theorem 4.23]{D1} that
\begin{equation}\label{r concavification}
E\in Int(L_p,L_q)\Longleftrightarrow E_{(r)}\in Int(L_{\frac pr},L_{\frac qr}).
\end{equation}
{\color{red}Similarly, let $E^{(r)}$ be $r$-convexification of $E$, that is, $\|f\|_{E^{(r)}}=\||f|^r\|_E^{1/r}.$}
\subsection{Orlicz functions and Orlicz spaces} Let $\Phi:\mathbb{R}\to\mathbb{R}_+$ be an Orlicz function, that is, $\Phi$ is an even convex function such that $\Phi(0)=0$ and $\Phi(\infty)=\infty.$ Given an Orlicz function $\Phi$ and an operator $x\in L_0(\M,\tau)$, we have by \cite[Corollary 2.8]{FK},
$$\tau(\Phi(|x|))=\int_0^\infty\Phi\big(\mu(t,x)\big)dt.$$ Obviously, if $\Phi(t)=t^p$ for $1\leq p<\infty$ then this reduces to the usual $p$-moment of $|x|.$ By induction, it follows from {\color{red}\cite[Theorem 4.4 (iii)]{FK}} that for every sequence $(x_i)_{i=0}^n\subset L_0(\M,\tau)$ and scalars $\lambda_i\in (0,1)$ with $\sum_{i=0}^n\lambda_i\leq1$,
\begin{equation}\label{convex phi}
\tau\Big(\Phi\Big(|\sum_{i=0}^n\lambda_ix_i|\Big)\Big)\leq \sum_{i=0}^n\lambda_i\tau\Big(\Phi(|x_i|)\Big).
\end{equation}
We now recall the definition of Orlicz spaces. Given an Orlicz function $\Phi$, the Orlicz function space $L_\Phi(0,\alpha)$, $0<\alpha\leq\infty$ is the set of all measurable functions $f$ on $(0,\alpha)$ such that
$$\|f\|_{L_\Phi}:=\inf\Big\{\lambda>0,\int_0^\infty\Phi\Big(\frac{|f(t)|}\lambda\Big)\leq1{\color{red}\Big\}}.$$
The Banach space $(L_\Phi(0,\alpha),\|\cdot\|_{L_\Phi})$ has the Fatou property (see \cite[p.64]{KPS}) and hence is a fully symmetric Banach space in the sense of \cite{DDP-Int}. An Orlicz function is called to satisfy $\Delta_2$-condition if there exists a constant $a>1$ such that $\Phi(at)\lesssim_a \Phi(t)$ for all $t>0.$ Note that if $\Phi$ satisfies $\Delta_2$-condition, then $f\in L_\Phi(0,\alpha)$ if and only if $\int_0^\infty\Phi(|f|)<\infty.$
If $\Phi$ satisfy the $\Delta_2$-condition, then by \eqref{convex phi} we have
\begin{equation}\label{quasi-trangle}
\tau(\Phi(x+y))\lesssim_\Phi\tau(\Phi(x))+\tau(\Phi(y)),\quad x,y \in L_\Phi(\M).
\end{equation}

Given $1\leq p\leq q\leq\infty$, an Orlicz function $\Phi$ is said to be $p$-convex if the function $t\to\Phi(t^{\frac 1p}), t>0$ is convex, and $\Phi$ is said to be $q$-concave if the function $t\to\Phi(t^{\frac1q}), t>0$ is concave. If $\Phi$ is $p$-convex and $q$-concave for $1\leq p\leq q\leq \infty,$ then the associated Orlicz space $L_\Phi$ is $p$-convex and is $q$-concave \cite[Theorem 50]{JSZ}, and hence the Boyd indices of $L_\Phi$ satisfy
$p\leq p_{L_\Phi}\leq q_{L_\Phi}\leq q$ (see \cite{LT} or \cite[Lemma 4.9]{D1}), which also implies that the standard {\rm Matuzewska-Orlicz indices } of the Orlicz function $\Phi$ satisfy $p\leq p_{\Phi}\leq q_{\Phi}\leq q$ \cite[Theorem 4.2]{Mali thesis}. It also follows from \cite[Theorem 3.2 (b)]{Mali thesis} that the $\Delta_2$-condition is equivalent to $q_\Phi<\infty.$

\subsection{Hardy-Littlewood preorder}
The following notion is used in the sequel.
\begin{definition} Let $\M$ be a semifinite von Neumann algebra and $x,y\in L_1(\M)+\M.$ We write $y\prec\prec x$ if
$$\int_0^t\mu(s,y)ds\leq\int_0^t\mu(s,x)ds,\quad t>0.$$
\end{definition}

It is proved in \cite[Theorem II.3.2]{KPS} that $y\prec\prec x$ if and only if there exists a $T:L_1+L_{\infty}\to L_1+L_{\infty}$ such that $T:L_1\to L_1,$ $T:L_{\infty}\to L_{\infty}$ are contractions and such that $Tx=y.$

We say that $E$ is fully symmetric if $\|\cdot\|_E$ is monotone with respect to $\prec\prec.$ The following result is due to Calder\'on and Mityagin (see \cite[Theorem II.3.4]{KPS}).

\begin{theorem} Symmetric quasi-Banach function space $E\in Int(L_1,L_{\infty})$ iff it is fully symmetric.
\end{theorem}

It follows from \cite[Theorem 11]{DSZ} that for every Orlicz function $\Phi$
\begin{equation}\label{majorization phi}
x\prec\prec y\Longrightarrow {\tau}(\Phi(x))\leq {\tau}(\Phi(y)),\quad \forall x,y\in L_1(\M)+\M.
\end{equation}

If $x,y,z\in L_1(\M)+\M$, then $x\otimes z, y\otimes z \in L_1(\M\bar{\otimes}\M)+\M\bar{\otimes} \M$. It follows from above that if $x,y,z\in L_1(\M)+\M,$ then
\begin{equation}\label{majorization direct}
y\prec\prec x\Longrightarrow y\otimes z\prec\prec x\otimes z.
\end{equation}

\subsection{Noncommutative matringales} A noncommutative probability space is a couple $(\mathcal{M},\tau)$ where $\M$ is a finite von Neumann algebra and {\color{red} $\tau$ is a normal faithful trace with $\tau(1)=1.$} Let $(\mathcal{M}_n)_{n\geq0}$ be an increasing sequence of von Neumann subalgebras of $\mathcal{M}$ such that the union of the $\mathcal{M}_n$'s is weak$^*$-dense in $\mathcal{M}.$ Let $\mathcal{E}_n$ be the conditional expectation of $\mathcal{M}$ with respect to $\mathcal{M}_n.$

\begin{definition} A sequence $x=(x_k)_{k\geq 0}$ in $L_1(\mathcal{M})$ is a sequence of martingale differences if $x_k\in\mathcal{M}_k$ for $k\geq0$ and if $\mathcal{E}_{k-1}(x_k)=0$ for every $k\geq1.$
\end{definition} In this paper, we always consider noncommutative martingales associated with a noncommutative probability space unless explicit explanation.

\section{Disjointification inequalities in symmetric operator spaces}\label{js section}

In this section we establish disjointification inequalities for noncommutative martingale difference sequences in symmetric operator spaces. We first state several lemmas.

\begin{lemma}\label{sum interpolation} Let $(\mathcal{M},\tau)$ be a finite von Neumann algebra and let $(\mathcal{N},\nu)$ be a semifinite atomless one. Suppose that $E\in Int(L_1(0,1), L_2(0,1))$ is a symmetric quasi-Banach space. If $T:L_2(\N)\to L_2(\M)$ and $T:L_1(\N)\to L_1(\M)$ is a linear contraction, then
$$\|Tx\|_{E(\M)}\lesssim_E\|x\|_{Z_E^2(\N)},\quad \forall x\in Z_E^2(\N).$$
\end{lemma}
\begin{proof} To see this, fix $x=x^*\in Z_E^2(\N)$ and consider the spectral projection of $|x|$ associated with the interval $(\mu(1,x),\infty)$,
$$p:=e_{(\mu(1,x),\infty)}(|x|).$$
Clearly, $\nu(p)\leq 1.$ Since $\N$ is atomless, we can fix a projection $q\in \N$ such that $q\geq p$ and $\nu(q)=1.$ By assumptions, $T:L_1(q\N q)\to L_1(\M)$ and $T:L_2(q\N q)\to L_2(\M)$ is a contraction. Since $E$ is an interpolation space between $L_1(0,1)$ and $L_2(0,1),$
 it follows from \cite[Theorem 3.2]{DDP-Int} that $T:E(q\N q)\to E(\M)$ is a bounded map. Since $\mu(t,qxq)=0$ for all $t$ greater than trace of $q,$ we have  $$\mu(qxq)=\mu(qxq)\chi_{(0,1)}\leq \mu(x)\chi_{(0,1)}.$$ On the other hand, $1-q\leq1-p$ and therefore by \eqref{1+2},
\begin{eqnarray*}\|qx(1-q)\|^2_{L_2(\N)}&\leq& \|x(1-q)\|^2_{L_2(\N)}=\nu(x(1-q)x)\leq\nu(x(1-p)x)
\\&=&\|x(1-p)\|^2_{L_2(\N)}=\int_{\{\mu(x)\leq\mu(1,x)\}}\mu^2(t,x)dt
\\&\leq&\|\mu(x)\chi_{(1,\infty)}\|_2^2+\mu^2(1,x)\lesssim\|\mu(x)\|^2_{L_1+L_2}.
\end{eqnarray*}
Similarly, $$\|(1-q)x\|^2_{L_2(\N)}\lesssim\|\mu(x)\|^2_{L_1+L_2}.$$
Consequently, taking into account that $\|\cdot\|_E\leq \|\cdot\|_2$ we have
\begin{eqnarray*}
\|Tx\|_{E(\M)}&\leq&\|T(qxq)\|_{E(\M)}+\|T(qx(1-q))\|_{E(\M)}+\|T((1-q)x)\|_{E(\M)}
\\&\leq&\|T(qxq)\|_{E(\M)}+\|T(qx(1-q))\|_{L_2(\M)}+\|T((1-q)x)\|_{L_2(\M)}
\\&\lesssim_E&\|qxq\|_{E(q\N q)}+\|qx(1-q)\|_{L_2(\N)}+\|(1-q)x\|_{L_2(\N)}
\\&\lesssim&\|\mu(x)\chi_{(0,1)}\|_E+\|\mu(x)\|_{L_1+L_2}\lesssim\|x\|_{ Z_E^2(\N)}.
\end{eqnarray*}
Hence,  $$\|Tx\|_{E(\M)}\leq \|T(\Re(x))\|_{E(\M)}+\|T(\Im(x))\|_{E(\M)}\lesssim_E\|x\|_{Z_E^2(\N)},\quad \forall x\in Z_E^2(\N).$$
\end{proof}

The following proposition extends the result of \cite[Lemma 3.5]{ASS} to noncommutative setting. The assumption that $E$ is a separable symmetric Banach space was used in \cite{ASS}.
\begin{proposition}\label{left estimate} Let $E\in Int(L_1(0,1),L_2(0,1))$ be a symmetric quasi-Banach space and let $(x_k)_{k\geq0}\subset E(\M)$ be a sequence of martingale differences. We have
$$\|\sum_{k\geq0}x_k\|_{E(\M)}\lesssim_E\|\sum_{k\geq0}x_k\otimes e_k\|_{Z_E^2(\M\bar{\otimes} \ell_\infty)}.$$
\end{proposition}
\begin{proof} {\color{red} Without loss of generality, we may assume that $\mathcal{M}$ is atomless. Indeed, otherwise, we consider algebra $\mathcal{M}\bar{\otimes} L_{\infty}(0,1)$ and a sequence $\{x_k\otimes 1\}_{k\geq0}.$}

Consider the operator
$$T:\sum_{k\geq0}x_k\otimes e_k\to\sum_{k\geq0}\Big(\mathcal E_k(x_k)-\mathcal{E}_{k-1}(x_k)\Big),\quad \forall\, (x_k)_{k\geq0}\subset (L_1+L_2)(\M).$$
We claim that $T$ is a bounded map from $L_2(\M\bar{\otimes}\ell_\infty)$ to $L_2(\M).$ In fact $\{\mathcal{E}_k(x_k)-\mathcal{E}_{k-1}(x_k)\}_{k\geq0}$ is an orthogonal sequence and, therefore,
$$\|T(\sum_{k\geq0}x_k\otimes e_k)\|_2^2=\sum_{k\geq0}\|\mathcal E_k(x_k)-\mathcal{E}_{k-1}(x_k)\|_2^2\leq 4\|\sum_{k\geq0}x_k\otimes e_k\|_2^2.$$
We also claim that $T:L_1(\M\bar{\otimes}\ell_\infty)\to L_1(\M)$ is bounded. Indeed, we have
$$\|T(\sum_{k\geq0}x_k\otimes e_k)\|_1\leq\sum_{k\geq0}\|\mathcal E_k(x_k)-\mathcal{E}_{k-1}(x_k)\|_1\leq 2\|\sum_{k\geq0}x_k\otimes e_k\|_1.$$
Using Lemma \ref{sum interpolation}, we obtain that $T:Z_E^2(\M\bar{\otimes}\ell_\infty)\to E(\M)$ and, moreover,
$$\|Tx\|_{E(\M)}\lesssim_E\|x\|_{Z_E^2(\M\bar{\otimes}\ell_\infty)},\quad x\in Z_E^2(\M\bar{\otimes}\ell_\infty).$$
In particular, if $(x_k)_{k\geq0}\subset E(\M)$ is a sequence of martingale differences, then
$$\|\sum_{k\geq0}x_k\|_{E(\M)}=\|T(\sum_{k\geq0}x_k\otimes e_k)\|_{E(\M)}\lesssim_E\|\sum_{k\geq0}x_k\otimes e_k\|_{Z_E^2(\M\bar{\otimes}\ell_\infty)}.$$
This completes the proof.
\end{proof}

\begin{lemma}\label{intersection interpolation} Let $(\mathcal{M},\tau)$ be a finite von Neumann algebra and let $(\mathcal{N},\nu)$ be a semifinite atomless one. Let $E\in Int(L_2(0,1),L_\infty(0,1))$ be a symmetric quasi-Banach space. If $T:L_2(\M)\to L_2(\N)$ and $T:L_{\infty}(\M)\to L_{\infty}(\N)$ is a linear contraction which maps self-adjoint operators to self-adjoint ones, then
$$\|Tx\|_{Z_E^2(\N)}\lesssim_E\|x\|_{E(\M)},\quad x\in E(\M).$$
\end{lemma}
\begin{proof}  To see this, fix $x=x^*\in E(\M)$ and consider the spectral projection of $|Tx|$ associated with the interval $(\mu(1,Tx),\infty)$,
$$p:=e_{(\mu(1,Tx),\infty)}(|Tx|).$$
It is clear that $\nu(p)\leq1.$ Fix a projection $q\in N$ such that $q\geq p$ and $\nu(q)=1.$ Consider the operator
$$S:z\to q\cdot Tz\cdot q,\quad z\in (L_2+L_\infty)(\M).$$
Using the assumption on $T$, we immediately obtain that $S:L_2(\M)\to L_2(q\N q)$ and $S:L_{\infty}(\M)\to L_{\infty}(q\N q)$ are also contractions. Since $E$ is an interpolation space between $L_2(0,1)$ and $L_{\infty}(0,1),$
 it follows from \cite[Theorem 3.2]{DDP-Int} that $S:E(\M)\to E(q\N q)$ is a bounded map.
 Noting that $\mu(q\cdot Tx\cdot q)=\mu(q\cdot Tx\cdot q)\chi_{(0,1)}$ and
 $$\|(1-q)\cdot Tx\|_{L_{\infty}(\N)}\leq \|(1-p)\cdot Tx\|_{L_{\infty}(\N)}\leq\mu(1,Tx),$$
 we have
\begin{eqnarray*}\|Tx\|_{Z_E^2(\N)}&\leq&\|q\cdot Tx\cdot q\|_{Z_E^2(\N)}+\|q\cdot Tx\cdot (1-q)\|_{Z_E^2(\N)}+\|(1-q)\cdot Tx\|_{Z_E^2(\N)}
\\&\leq&\|Sx\|_{E(q\N q)}+\|q\cdot Tx\cdot q\|_2+\|q\cdot Tx\cdot (1-q)\|_{(L_2\cap L_{\infty})(\N)}
\\&+&\|(1-q)\cdot Tx\|_{(L_2\cap L_{\infty})(\N)}
\leq\|Sx\|_{E(q\N q)}+3\|Tx\|_{L_2(\N)}\\&+&2\|(1-q)\cdot Tx\|_{L_{\infty}(\N)}\lesssim_E\|Sx\|_{E(q\N q)}+3\|Tx\|_{L_2(\N)}+2\mu(1,Tx)\\&\leq& \|Sx\|_{E(q\N q)}+5\|Tx\|_{L_2(\N)}.
\end{eqnarray*}
Thus,
$$\|Tx\|_{Z_E^2(\N)}\lesssim_E\|x\|_{E(\M)}+5\|x\|_{L_2(\M)}\leq 6\|x\|_{E(\M)},\,\quad x=x^*\in E(\M).$$
By splitting $x$ into its real part and imaginary parts, we conclude the proof.
\end{proof}

\begin{remark}\label{q} Suppose that in Lemma \ref{intersection interpolation}, we have that $E\in Int(L_2(0,1),L_q(0,1))$ is a symmetric quasi-Banach space for some $2<q<\infty$, and that $T:L_2(\M)\to L_2(\N)$ and $T:L_{q}(\M)\to L_{q}(\N)$ are linear contractions which map self-adjoint operators to self-adjoint ones. In this case, the result of Lemma \ref{intersection interpolation} still holds. Indeed, we just need to notice that
\begin{eqnarray*}\|(1-q)\cdot Tx\|_{L_{q}(\N)}&\leq& \|(1-p)\cdot Tx\|_{L_{q}(\N)}\\&\leq&\Big(\int_1^\infty\mu^q(t,Tx)dt\Big)^{1/q}+\mu(1,Tx) \\&\lesssim&\|Tx\|_{(L_2+L_q)(\N)}\leq\|x\|_{(L_2+L_q)(\M)}\\&=&\|x\|_{L_2(\M)}\leq \|x\|_{E(\M)}.
\end{eqnarray*}
\end{remark}
The following proposition extends Theorem 3.5 and \cite[Corollary 3.6]{ASW}.

\begin{proposition}\label{right estimate} Let $E\in {\color{red}Int(L_2(0,1),L_\infty(0,1))}$ be a symmetric quasi-Banach space and let $\{x_k\}_{k\geq0}\subset E(\M)$ be a sequence of martingale differences. We have
$$\|\sum_{k\geq0}x_k\otimes e_k\|_{Z_E^2(\M{\bar\otimes} \ell_\infty)}\lesssim_E\|\sum_{k\geq0}x_k\|_{E(\M)}.$$
\end{proposition}
\begin{proof} Let $\{\mathcal{M}_k\}_{k\geq0}\subset\mathcal{M}$ be an increasing sequence of unital von Neumann subalgebras and let $\mathcal{E}_k:\mathcal{M}\to\mathcal{M}_k$ be the conditional expectations. For brevity, we set $\mathcal{E}_{-1}=0.$ Consider the operator
$$T:x\to\sum_{k\geq0}\Big(\mathcal{E}_k(x)-\mathcal{E}_{k-1}(x)\Big)\otimes e_k,\quad x\in L_2(\M).$$
We have $T:L_2(\M)\to L_2(\M{\bar\otimes}\ell_{\infty})$ (in fact, $\|Tx\|_2=\|x\|_2$ for every $x\in L_2(\M)$). It is immediate that $\|Tx\|_{\infty}\leq 2\|x\|_{\infty}.$ It follows from Lemma \ref{intersection interpolation} that $T$ boundedly maps $E(\M)$ into $Z_E^2(\M{\bar\otimes}\ell_{\infty}).$ Applying this result to the element $x=\sum_{k\geq0}x_k,$  we have $$Tx=\sum_{k\geq0}\Big(\mathcal{E}_k(x)-\mathcal{E}_{k-1}(x)\Big)\otimes e_k=\sum_{k\geq0}x_k\otimes e_k$$ and the assertion immediately follows.
\end{proof}

Our first main result Theorem \ref{first main} now follows immediately from the combination of Proposition \ref{left estimate} and Proposition \ref{right estimate}.

\begin{corollary}\label{first cor} Let $E$ be a symmetric quasi-Banach space on $(0,1)$. Let $(x_k)_{k=0}^\infty\subset E(\M)$  be an arbitrary sequence and $(r_k)_{k=0}^\infty$ be the Rademacher sequence on $(0,1).$
\begin{enumerate}[{\rm (i)}]
\item\label{fmca} If $E\in Int(L_1(0,1),L_2(0,1)),$ then
$$\|\sum_{k\geq0}x_k\otimes r_k\|_{E(\M\bar{\otimes} L_{\infty}(0,1))}\lesssim_E\|\sum_{k\geq0}x_k\otimes e_k\|_{Z_E^2(\M{\bar\otimes} \ell_\infty)}.$$
\item\label{fmcb} If $E\in Int(L_2(0,1),L_{\infty}(0,1)),$ then
$$\|\sum_{k\geq0}x_k\otimes e_k\|_{Z_E^2(\M{\bar\otimes} \ell_\infty)}\lesssim_E\|\sum_{k\geq0}x_k\otimes r_k\|_{E(\M\bar{\otimes} L_{\infty}(0,1))}.$$
\end{enumerate}
\end{corollary}

 We now prove Theorem \ref{js 2 sides}. Let $E$ be a symmetric quasi-Banach space. For $s>0$, define $\sigma_s:E\rightarrow E$ by setting (see \cite{KPS})
$$\sigma_sf(t)=f(t/s),\quad t>0,\quad f\in E.$$

\begin{proof}[Proof of Theorem \ref{js 2 sides}] Recall that the classical Haar system $(h_k)_{k\geq0}$ is defined as follows. First, set $h_0=1.$ If $k=2^m+l,$ $0\leq l<2^m,$ then
$$h_k=\chi_{(\frac{2l}{2^{m+1}},\frac{2l+1}{2^{m+1}})}-\chi_{(\frac{2l+1}{2^{m+1}},\frac{2l+2}{2^{m+1}})}.$$
Define a filtration\footnote{We took the filtration from the Proposition 6.1.3 in \cite{AK}.} $(\mathcal{M}_k)_{k\geq0}$ in $L_{\infty}(0,1)$ as follows. First, set $\mathcal{M}_0=\mathbb{C}.$ If $k=2^m+l,$ $0\leq l<2^m,$ then $$\mathcal{M}_k={\rm span}\big\{(\chi_{(\frac{j}{2^{m+1}},\frac{j+1}{2^{m+1}})})_{j=0}^{2l+1},\ (\chi_{(\frac{j}{2^m},\frac{j+1}{2^m})})_{j=l+1}^{2^m-1}\big\} $$
It is immediate that $(h_k)_{k\geq0}$ is a sequence of martingale differences in $L_{\infty}(0,1)$ with respect to the filtration $(\mathcal{M}_k)_{k\geq0}.$

Now we apply equality \eqref{js equality} to the sequences  $(\alpha_kh_k)_{k\geq0}$ and $(|\alpha_k|h_k)_{k\geq0}$ of martingale differences (here, $(\alpha_k)_{k\geq0}$ is a sequence of scalars). It is immediate that
$$\|\sum_{k\geq0}\alpha_kh_k\|_{E(0,1)}\approx_E\Big\|\sum_{k\geq0}\alpha_k h_k\otimes e_k\Big\|_{Z_E^2({\color{red}L_\infty(0,1)\bar{\otimes}\ell_\infty)}}\approx_E\|\sum_{k\geq0}|\alpha_k|h_k\|_{E(0,1)}.$$
Thus, Haar system is unconditional in $E.$ According to \cite[Theorem 2.c.6]{LT}, we have that $E\in Int(L_p,L_q)$ for some $1<p<q<\infty.$

Now, fix $x\in L_{\infty}(0,1)$ and consider the sequence $(x\otimes r_k)_{k=0}^{n-1}.$ Since the latter sequence consists of martingale differences, it follows from \eqref{js equality} that
$$\Big\|x\otimes\sum_{k=0}^{n-1}r_k\Big\|_E\approx_E\|\sigma_n\mu(x)\|_{Z_E^2}\approx\|(\sigma_n\mu(x))\chi_{(0,1)}\|_E+\|(\sigma_n\mu(x))\chi_{(1,\infty)}\|_2,$$
{\color{red}where the first $"\approx"$ follows from a simple fact that
$$\mu\Big(\bigoplus_{k=0}^{n-1}(x\otimes r_k)\Big)=\sigma_n\mu(x).$$}
Note that $$\|(\sigma_n\mu(x))\chi_{(0,1)}\|_E\leq\|x\|_\infty\quad {\rm and}\quad \|(\sigma_n\mu(x))\chi_{(1,\infty)}\|_2=n^{1/2}\Big(\int_{\frac1n}^\infty\mu^2(t,x)dt\Big)^{1/2}.$$
Dividing by $n^{1/2}$, setting $x_n=\frac1{n^{1/2}}\sum_{k=0}^{n-1}r_k$ and passing $n\to\infty,$ we obtain that
$$\limsup_{n\to\infty}\|x\otimes x_n\|_E\approx_E \|x\|_2.$$
Since for {\color{red}every} $0<t\leq 1,$
$$\|x_n\|_1=\frac1t\int_0^t\mu(\frac st,x_n)ds\leq \frac1t\int_0^t\mu(s,x_n)ds,$$
we get $\|x_n\|_1\prec\prec  x_n$, and hence by \eqref{majorization direct},
$$x\otimes\|x_n\|_1\prec \prec x\otimes x_n\quad {\rm and}\quad \|x\otimes\|x_n\|_1\|_E\lesssim_E\|x\otimes x_n\|_E.$$
Therefore, by Khinchine inequality for $L_1,$
\begin{equation}\label{E2 right}
\|x\|_E\approx\limsup_{n\to\infty}\|x\|_E\|x_n\|_1\leq\limsup_{n\to\infty}\|x\otimes x_n\|_E\approx_E\|x\|_2.
\end{equation}
On the other hand, the operators $x\to x\otimes x_n$ are uniformly bounded in $L_1$ and in $L_q$  (by Khinchine inequality). Hence, these operators are uniformly bounded in $E$ because that $E\in Int(L_1,L_q).$ Thus,
\begin{equation}\label{E2 left}
\|x\|_2\approx_E\limsup_{n\to\infty}\|x\otimes x_n\|_E\lesssim_q \|x\|_E.
\end{equation}
Combining \eqref{E2 right} and \eqref{E2 left}, we get
\begin{equation}\label{E2}
\|x\|_2\approx_E \|x\|_E,\quad x\in L_{\infty}(0,1).
\end{equation}

We now can omit the assumption $x\in L_{\infty}(0,1).$ Firstly for every $x\in E,$ it follows from \eqref{E2} that
$$\|x\|_E\geq \|\mu(x)\chi_{(\frac1n,1)}\|_E\approx_E\|\mu(x)\chi_{(\frac1n,1)}\|_2,\quad \forall n\geq1.$$
Passing $n\to\infty$, we get $\|x\|_2\lesssim\|x\|_E$ and, therefore, $E\subset L_2.$ If $x\in L_2$, then there exists $(x_k)\subset L_\infty(0,1)$ such that $\|x_k-x\|_2\to0.$ Therefore, again by \eqref{E2}
$$\|x_k-x_l\|_E\approx_E\|x_k-x_l\|_2\to0,\quad k,l\to\infty,$$
which implies that $x_k\to x$ in $E$. Thus, $x\in E$ and we conclude the proof.
\end{proof}

\section{Burkholder-Gundy inequality in symmetric Banach spaces}\label{bg section}

The main result in this section is Theorem \ref{second main}, which extends the Pisier-Xu noncommutative Burkholder-Gundy inequality \cite[Theorem 2.1]{PX} and also improves \cite[Theorem 3.1]{J-BG} and \cite[Proposition 4.18]{DPPS}. We first state several lemmas.

\begin{lemma}\label{left hinchine estimate} Let $E$ be a symmetric Banach space on $(0,1)$ and let $(x_k)_{k\geq0}\subset E(\M) $ be an arbitrary sequence. If $E\in Int(L_1(0,1),L_2(0,1)),$ then
$$\|\sum_{k\geq0}x_k\otimes r_k\|_{E(\M\bar{\otimes} L_{\infty}(0,1))}\lesssim_{ E}\|(\sum_{k\geq0}|x_k|^2)^{1/2}\|_{E(\M)}.$$
\end{lemma}
\begin{proof} {\color{red} Without loss of generality, we may assume that $\mathcal{M}$ is atomless. Indeed, otherwise, we consider algebra $\mathcal{M}\bar{\otimes} L_{\infty}(0,1)$ and a sequence $\{x_k\otimes 1\}_{k\geq0}.$}

Recalling that the column subspace of $L_p(\M\bar{\otimes}B(\ell_2))$ is a $1$-complemented subspace for every $0<p\leq\infty$, we consider the following linear mapping defined on $(L_1+L_2)(\M\bar{\otimes}B(\ell_2))$
$$T:\sum_{k\geq0}x_k\otimes e_{k1}\to\sum_{k\geq0}x_k\otimes r_k.$$
Since for every $x=\sum_{i,j\geq0}x_{ij}\otimes e_{ij}\in L_2(\M\bar{\otimes}B(\ell_2))$,
$$\|\sum_{i\geq0}x_{i1}\otimes r_i\|_{L_2(\M\bar{\otimes} L_{\infty}(0,1))}=\|\sum_{i\geq0}x_{i1}\otimes e_{i1}\|_2=\|\sum_{i,j\geq0}(x_{ij}\otimes e_{ij})(1\otimes e_{11})\|_2\leq\|x\|_2.$$
We now have a linear bounded operator $S:L_2(\M\bar{\otimes}B(\ell_2))\to L_2(\M\bar{\otimes} L_{\infty}(0,1))$ defined by the setting
$$Sx=T(x\cdot (1\otimes e_{11})),\quad x\in L_2(\M\bar{\otimes}B(\ell_2)).$$
It follows from noncommutative Khinchine inequality for $L_1(\M)$ (see \cite{LPP}) that for any $x=\sum_{i,j\geq0}x_{ij}\otimes e_{ij}\in L_1(\M\bar{\otimes}B(\ell_2)),$
$$\|Sx\|_{L_1(\M\bar{\otimes} L_{\infty}(0,1))}\leq \|(\sum_{k\geq0}|x_{i1}|^2)^{1/2}\|_{L_1(\M)}\leq\|\sum_{i\geq0}x_{i1}\otimes e_{i1}\|_1\leq\|x\|_1.$$
Applying Lemma \ref{sum interpolation}, we obtain
$$\|Sx\|_{E(\M\bar{\otimes} L_{\infty}(0,1)))}\lesssim_{ E}\|x\|_{Z_E^2(\M\bar{\otimes}B(\ell_2))}.$$
Setting $x=\sum_{k\geq0}x_k\otimes e_{k1},$ we get
$$\|\sum_{k\geq0}x_k\otimes r_k\|_{E(\M\bar{\otimes} L_{\infty}(0,1))}\lesssim_{ E}\|\sum_{k\geq0}x_k\otimes e_{k1}\|_{Z_E^2(\M\bar{\otimes}B(\ell_2))}.$$
However,
$$|\sum_{k\geq0}x_k\otimes e_{k1}|=\Big(\sum_{k\geq0}|x_k|^2\Big)^{1/2}\otimes e_{11}$$
and, therefore,
\begin{eqnarray*}\|\sum_{k\geq0}x_k\otimes r_k\|_{E(\M\bar{\otimes} L_{\infty}(0,1))}&\lesssim_{ E}&\Big\|\Big(\sum_{k\geq0}|x_k|^2\Big)^{1/2}\otimes e_{11}\Big\|_{Z_E^2(\M\bar{\otimes}B(\ell_2))}\\&=&\Big\|\Big(\sum_{k\geq0}|x_k|^2\Big)^{1/2}\Big\|_{E(\M)}
+\Big\|\Big(\sum_{k\geq0}|x_k|^2\Big)^{1/2}\Big\|_{(L_1+L_2)(\M)}
\\ &\leq&2\Big\|\Big(\sum_{k\geq0}|x_k|^2\Big)^{1/2}\Big\|_{E(\M)}.
\end{eqnarray*}
\end{proof}

\begin{lemma}\label{right hinchine estimate} Let $E$ be a symmetric Banach space  on $(0,1)$ and let $(x_k)_{k\geq0}\subset E(\M)$ be a sequence of martingale differences. If $E\in Int(L_2(0,1),L_q(0,1))$ for some $2<q<\infty,$ then
$$\Big\|\Big(\sum_{k\geq0}|x_k|^2\Big)^{1/2}\Big\|_{E(\M)}\lesssim_{E}\|\sum_{k\geq0}x_k\|_{E(\M)}.$$
\end{lemma}
\begin{proof} Consider the linear mapping
$$T:x\to\sum_{k\geq0}(\mathcal{E}_k(x)-\mathcal{E}_{k-1}(x))\otimes e_{k1},\quad x\in L_q(\M).$$
Clearly, $T:L_2(\M)\to L_2(\M\bar{\otimes}B(\ell_2))$ (moreover, $\|Tx\|_{L_2(\M\bar{\otimes}B(\ell_2))}=\|x\|_{L_2(\M)}$). Using \cite[Theorem 2.1]{PX} (see also \cite[Theorem 5.1]{N}), we obtain
$$\|Tx\|_{L_q(\M\bar{\otimes}B(\ell_2))}=\Big\|\Big(\sum_{k\geq0}|\mathcal{E}_k(x)-\mathcal{E}_{k-1}(x)|^2\Big)^{1/2}\Big\|_{L_q(\M)}\lesssim_q\|x\|_{L_q(\M)}.$$
By Lemma \ref{intersection interpolation} and Remark \ref{q}, we have that
$\|Tx\|_{Z_E^2(\M\bar{\otimes}B(\ell_2))}\lesssim_E\|x\|_{E(\M)}.$
Again, $$\big|\sum_{k\geq0}(\mathcal{E}_k(x)-\mathcal{E}_{k-1}(x))\otimes e_{k1}\big|=\Big(\sum_{k\geq0}|\mathcal{E}_k(x)-\mathcal{E}_{k-1}(x)|^2\Big)^{1/2}\otimes e_{11},$$
therefore,
$$\Big\|\Big(\sum_{k\geq0}|\mathcal{E}_k(x)-\mathcal{E}_{k-1}(x)|^2\Big)^{1/2}\Big\|_{E(\M)}\lesssim_E\|x\|_{E(\M)}.$$
Substituting $x=\sum_{k\geq0}x_k,$ we conclude the proof.
\end{proof}

The following lemma is attributed in \cite{DPPS} to Bekjan \cite{Bekjan-rocky}. However, $p$-convexity and/or $q$-concavity assumed everywhere in \cite{Bekjan-rocky} makes result of that paper unsuitable for our purposes. We present a different proof based on the results from \cite{J-BG}.

\begin{lemma}\label{bekjan} Let $E$ be a symmetric Banach space on $(0,1)$ and let $(x_k)_{k\geq0}\subset E(\M)$ be a sequence of martingale differences. If $E\in Int(L_p,L_2)$ for $1<p<2,$ then
\begin{eqnarray*}
&&\inf_{x_k=y_k+z_k}\Big(\Big\|\Big(\sum_{k\geq0}|y_k|^2\Big)^{1/2}\Big\|_{E(\M)}+\Big\|\Big(\sum_{k\geq0}|z_k^*|^2\Big)^{1/2}\Big\|_{E(\M)}\Big)
\\&\approx_E&\inf_{\substack{x_k=y_k+z_k\\ \{y_k\}_{k\geq0},\{z_k\}_{k\geq0}\mbox{ are}\\ \mbox{martingale differences}}}\Big(\Big\|\Big(\sum_{k\geq0}|y_k|^2\Big)^{1/2}\Big\|_{E(\M)}+\Big\|\Big(\sum_{k\geq0}|z_k^*|^2\Big)^{1/2}\Big\|_{E(\M)}\Big).
\end{eqnarray*}
\end{lemma}
\begin{proof} Let $x_k=y_k+z_k.$ Since $(x_k)_{k\geq0}$ is a sequence of martingale differences, it follows that
$$x_k=\mathcal{E}_k(x_k)-\mathcal{E}_{k-1}(x_k)=(\mathcal{E}_k(y_k)-\mathcal{E}_{k-1}(y_k))+(\mathcal{E}_k(z_k)-\mathcal{E}_{k-1}(z_k))\stackrel{def}{=}u_k+v_k.$$
Clearly, $(u_k)_{k\geq0}$ and $(v_k)_{k\geq0}$ are sequences of martingale differences. We have
$$\|\Big(\sum_{k\geq0}|u_k|^2\Big)^{1/2}\|_{E(\M)}\leq \|\Big(\sum_{k\geq0}|\mathcal{E}_k(y_k)|^2\Big)^{1/2}\|_{E(\M)}+\|\Big(\sum_{k\geq0}|\mathcal{E}_{k-1}(y_k)|^2\Big)^{1/2}\|_{E(\M)}.$$
It follows now from the noncommutative Stein inequality (see \cite[Lemma 3.3]{J-BG}) that
$$\|\Big(\sum_{k\geq0}|u_k|^2\Big)^{1/2}\|_{E(\M)}\lesssim_E \|\Big(\sum_{k\geq0}|y_k|^2\Big)^{1/2}\|_{E(\M)}.$$
Similarly, we have
$$\|\Big(\sum_{k\geq0}|v_k^*|^2\Big)^{1/2}\|_{E(\M)}\lesssim_E \|\Big(\sum_{k\geq0}|z_k^*|^2\Big)^{1/2}\|_{E(\M)}.$$
Hence, the infimum can be taken over all the sequences of martingale differences.
\end{proof}

We are now prepared to present the proof of Theorems \ref{second main} and \ref{second l2}.
\begin{proof}[Proof of Theorem \ref{second main}] Since $E\in Int(L_p,L_q)$ for $1<p<q<\infty,$ it follows from \cite[Lemma 4.17]{DPPS} that
\begin{equation}\label{arbitrary to tensor}
\|\sum_{k\geq0}x_k\|_{E(\M)}\approx_{E}\|\sum_{k\geq0}x_k\otimes r_k\|_{E(\M\bar{\otimes} L_{\infty}(0,1))}.
\end{equation}

\eqref{sma} {\color{red}\cite[Theorem 1.1]{lM-Suk} states that if $E$ is separable with finite upper Boyd index, then we have
$$\|\sum_{k\geq0}x_k\otimes r_k\|_{E(\M\bar{\otimes}L_{\infty}(0,1))}\gtrsim_{E}\inf_{x_k=y_k+z_k}\Big(\Big\|\Big(\sum_{k\geq0}|y_k|^2\Big)^{1/2}\Big\|_{E(\M)}+\Big\|\Big(\sum_{k\geq0}|z_k^*|^2\Big)^{1/2}\Big\|_{E(\M)}\Big).$$}
It immediately follows from \eqref{arbitrary to tensor} that
$$\|\sum_{k\geq0}x_k\|_{E(\M)}\gtrsim_{E}\inf_{x_k=y_k+z_k}\Big(\Big\|\Big(\sum_{k\geq0}|y_k|^2\Big)^{1/2}\Big\|_{E(\M)}+\Big\|\Big(\sum_{k\geq0}|z_k^*|^2\Big)^{1/2}\Big\|_{E(\M)}\Big),$$
where the infimum taken over all possible decompositions $x_k=y_k+z_k.$ By Lemma \ref{bekjan}, we have
$$\|\sum_{k\geq0}x_k\|_{E(\M)}\gtrsim_{E}\inf_{x_k=y_k+z_k}\Big(\Big\|\Big(\sum_{k\geq0}|y_k|^2\Big)^{1/2}\Big\|_{E(\M)}+\Big\|\Big(\sum_{k\geq0}|z_k^*|^2\Big)^{1/2}\Big\|_{E(\M)}\Big),$$
where the infimum is taken over the sequences of martingale differences.
We now prove the converse inequality. Let $x_k=y_k+z_k,$ with $(y_k)_{k\geq0}\subset E(\M)$ and $(z_k)_{k\geq0}\subset E(\M)$ being sequences of martingale differences. Applying \eqref{arbitrary to tensor} and Lemma \ref{left hinchine estimate}, we conclude that
$$\|\sum_{k\geq0}y_k\|_{E(\M)}\lesssim_{E}\Big\|\Big(\sum_{k\geq0}|y_k|^2\Big)^{1/2}\Big\|_{E(\M)}$$
and, similarly,
$$\|\sum_{k\geq0}z_k\|_{E(\M)}\lesssim_{E}\Big\|\Big(\sum_{k\geq0}|z_k^*|^2\Big)^{1/2}\Big\|_{E(\M)}.$$
Therefore, we have
$$\|\sum_{k\geq0}x_k\|_{E(\M)}\lesssim_{E}\inf_{x_k=y_k+z_k}\Big(\Big\|(\sum_{k\geq0}|y_k|^2\Big)^{1/2}\Big\|_{E(\M)}+\Big\|\Big(\sum_{k\geq0}|z_k^*|^2\Big)^{1/2}\Big\|_{E(\M)}\Big),$$
where the infimum is taken over the sequences of martingale differences.

\eqref{smb} {\color{red}\cite[Theorem 4.1]{DPPS} states that if $E$ is $p$-convex for some $0<p<\infty$ (we take $p=1$ since every Banach space is $1$-convex) and has finite upper Boyd index, then we have
$$\|\sum_{k\geq0}x_k\otimes r_k\|_{E(\M\bar{\otimes} L_{\infty}(0,1))}\lesssim_{E}\max\Big\{\Big\|\Big(\sum_{k\geq0}|x_k|^2\Big)^{1/2}\Big\|_{E(\M)},\Big\|\Big(\sum_{k\geq0}|x_k^*|^2\Big)^{1/2}\Big\|_{E(\M)}\Big\}.$$
Now it immediately follows from \eqref{arbitrary to tensor} that
$$\|\sum_{k\geq0}x_k\|_{E(\M)}\lesssim_{E}\max\Big\{\Big\|\Big(\sum_{k\geq0}|x_k|^2\Big)^{1/2}\Big\|_{E(\M)},\Big\|\Big(\sum_{k\geq0}|x_k^*|^2\Big)^{1/2}\Big\|_{E(\M)}\Big\}.$$}
The inequality
$$\|\sum_{k\geq0}x_k\|_{E(\M)}\gtrsim_{E}\max\Big\{\Big\|\Big(\sum_{k\geq0}|x_k|^2\Big)^{1/2}\Big\|_{E(\M)},\Big\|\Big(\sum_{k\geq0}|x_k^*|^2\Big)^{1/2}\Big\|_{E(\M)}\Big\}.$$
follows from Lemma \ref{right hinchine estimate}.
\end{proof}

\begin{proof}[Proof of Theorem \ref{second l2}] Let the Haar system $(h_k)_{k\geq0}$ and filtration $(\mathcal{M}_k)_{k\geq0}$ in $L_{\infty}(0,1)$ be defined as in the proof of Theorem \ref{js 2 sides}. Recall that $(h_k)_{k\geq0}$ is a sequence of martingale differences in $L_{\infty}(0,1)$ with respect to the filtration $(\mathcal{M}_k)_{k\geq0}.$

Now we apply the equality \eqref{bgundy leq2} (or \eqref{bgundy geq2}) to the sequences  $(\alpha_kh_k)_{k\geq0}$ and $(|\alpha_k|h_k)_{k\geq0}$ of martingale differences (here, $(\alpha_k)_{k\geq0}$ is a sequence of scalars). It is immediate that
$$\|\sum_{k\geq0}\alpha_kh_k\|_E\approx_E\|(\sum_{k\geq0}|\alpha_k|^2|h_k|^2)^{1/2}\|_E\approx_E\|\sum_{k\geq0}|\alpha_k|h_k\|_E.$$
Thus, Haar system is an unconditional basic sequence in $E.$ According to \cite[Theorem 2.c.6]{LT}, we have that $E\in Int(L_p,L_q)$ for $1<p<q<\infty.$
\end{proof}

We refer to \cite{J-D} for the maximal functions of noncommutative martingales and to \cite{dirksen2015} for the notation $E(\M;\ell_\infty).$ The assertion below follows from Theorem \ref{second main} above and \cite[Theorem 5.7]{dirksen2015}. Its dual version is also true for the case $E\in Int(L_p(0,1),L_2(0,1))$ for some $1<p<2.$ These results improve \cite[Theorem 6.1]{dirksen2015}.

\begin{corollary} Let $E$ be a symmetric Banach space on $(0,1)$ and let $(x_k)_{k\geq0}\subset E(\M)$ be a sequence of martingale differences. If $E\in Int(L_2(0,1),L_q(0,1))$ for some $2<q<\infty,$ then
$$\|(\sum_{l=0}^kx_l)_{k\geq0}\|_{E(\M;\ell_\infty)}\approx_{E}\max\Big\{\Big\|\Big(\sum_{k\geq0}|x_k|^2\Big)^{1/2}\Big\|_{E(\M)},\Big\|\Big(\sum_{k\geq0}|x_k^*|^2\Big)^{1/2}\Big\|_{E(\M)}\Big\}.$$
\end{corollary}

\section{Burkholder inequality in symmetric Banach spaces}\label{burk section}

Applying Theorem \ref{first main} and Theorem \ref{second main} proved in the preceding sections, in this section we partly resolve one problem stated in \cite{NW}. We refer to \cite[Theorem 3.3.6]{LSZ} for the following lemma.
\begin{lemma}\label{LSZ-th} Let $\M$ be a semifinite von Neumann algebra and let $0\leq A, B\in (L_1+L_\infty)(\M)$. Then $B\prec\prec A$ if and only if
$$\tau\big((B-t)_+\big)\leq \tau\big((A-t)_+\big),\quad \forall t>0.$$
\end{lemma}
\begin{corollary}\label{disjoint submajorization} Let $\M$ be a semifinite von Neumann algebra and let $A_k,B_k\in (L_1+L_\infty)(\M),$ $k\geq0.$ If $B_k\prec\prec A_k$ for all $k\geq0,$ then
$$\sum_{k\geq0}B_k\otimes e_k\prec\prec\sum_{k\geq0}A_k\otimes e_k.$$
\end{corollary}
\begin{proof} Without loss of generality, $A_k,B_k\geq0.$ It follows from the assumption and Lemma \ref{LSZ-th} that
$$\tau((B_k-t)_+)\leq\tau((A_k-t)_+),\quad t>0,$$
for every $k\geq0.$ Thus,\footnote{Here, $\#$ denotes the counting measure on $\mathbb{Z}_+,$ so that $\tau\otimes\#$ is a trace on the von Neumann algebra $\mathcal{M}\bar{\otimes} l_{\infty}.$}
\begin{eqnarray*}(\tau\otimes\#)((\sum_{k\geq0}B_k\otimes e_k-t)_+)&=&\sum_{k\geq0}\tau((B_k-t)_+)\leq\sum_{k\geq0}\tau((A_k-t)_+)
\\&=&(\tau\otimes\#)((\sum_{k\geq0}A_k\otimes e_k-t)_+),\, t>0.
\end{eqnarray*}
Again applying Lemma \ref{LSZ-th}, we conclude the proof.
\end{proof}

The following lemma is taken from \cite[Theorem 1]{M2013}.
\begin{lemma} \label{p-concavification}
Let $E$ and $F$ be two symmetric quasi-Banach function spaces and $1<p<\infty.$ We have
$$(E+F)^{(p)}=E^{(p)}+F^{(p)},\quad {\rm and}\quad (E+F)_{(p)}=E_{(p)}+F_{(p)}. $$
\end{lemma}

The following proposition contains crucial technical estimate needed for the proof of Theorem \ref{the 2-4}.
\begin{proposition}\label{lemma 2-4} Let $E\in Int (L_2,L_4)$ be a symmetric quasi-Banach space and let $(x_k)_{k\geq0}\subset E(\M)$ be a sequence of martingale differences. We have
$$\Big\|\Big(\sum_{k\geq0}|x_k|^2\Big)^{1/2}\Big\|_{E(\M)}\lesssim_E\Big\|\sum_{k\geq0}x_k\otimes e_k\Big\|_{Z_E^2(\M\bar{\otimes} \ell_\infty)}+\Big\|\Big(\sum_{k\geq0}\mathcal{E}_{k-1}(|x_k|^2)\Big)^{1/2}\Big\|_{E(\M)}.$$
\end{proposition}
\begin{proof} If $E\in Int(L_2,L_4),$ then it follows from \eqref{r concavification} that $E_{(2)}\in Int(L_1,L_2).$ We have
\begin{eqnarray*}
\Big\|\Big(\sum_{k\geq0}|x_k|^2\Big)^{1/2}\Big\|_{E(\M)}&=&\Big\|\sum_{k\geq0}|x_k|^2\Big\|_{E_{(2)}(\M)}^{1/2}
\\&=&\Big\|\sum_{k\geq0}\mathcal{E}_{k-1}(|x_k|^2)+\sum_{k\geq0}|x_k|^2-\mathcal{E}_{k-1}(|x_k|^2)\Big\|_{E_{(2)}(\M)}^{1/2}.
\end{eqnarray*}
By the quasi-triangle inequality in $E_{(2)}(\M)$, we have that
$$\Big\|\Big(\sum_{k\geq0}|x_k|^2\Big)^{1/2}\Big\|_{E(\M)}\lesssim_E\Big\|\sum_{k\geq0}\mathcal{E}_{k-1}(|x_k|^2)\Big\|_{E_{(2)}(\M)}^{1/2}+\Big\|\sum_{k\geq0}|x_k|^2-\mathcal{E}_{k-1}(|x_k|^2)\Big\|_{E_{(2)}(\M)}^{1/2}.$$
Since $\{|x_k|^2-\mathcal{E}_{k-1}(|x_k|^2)\}_{k\geq0}$ is a sequence of martingale differences, it follows from Theorem \ref{first main} \eqref{fma} that
\begin{multline}\label{4a}
\Big\|\Big(\sum_{k\geq0}|x_k|^2\Big)^{1/2}\Big\|_{E(\M)}\lesssim_E\Big\|\Big(\sum_{k\geq0}\mathcal{E}_{k-1}(|x_k|^2)\Big)^{1/2}\Big\|_{E(\M)}
\\+\Big\|\sum_{k\geq0}\Big(|x_k|^2-\mathcal{E}_{k-1}(|x_k|^2)\Big)\otimes e_k\Big\|_{Z_{E_{(2)}}^2(\M\bar{\otimes} \ell_\infty)}^{1/2}.
\end{multline}
Since every conditional expectation operator is a contraction on $L_1(0,1)$ and $L_\infty(0,1)$,  it follows from \cite[Theorem II.3.4]{KPS} (see also \cite[Lemma 3.6.2]{LSZ}) that
$$\mathcal{E}_{k-1}(|x_k|^2)\prec\prec|x_k|^2,\quad \forall k\geq0,$$
and therefore, by Corollary \ref{disjoint submajorization} we obtain
$$\sum_{k\geq0}\mathcal{E}_{k-1}(|x_k|^2)\otimes e_k\prec\prec\sum_{k\geq0}|x_k|^2\otimes e_k.$$
Since $E_{(2)}\in Int(L_1,L_2)$, it follows that $E_{(2)}\in Int(L_1,L_\infty).$ That is, up to a constant,  $E_{(2)}$ (and, hence, $Z_{E_{(2)}}^2$) is a fully symmetric quasi-Banach space. Thus,
\begin{equation}\label{4b}\Big\|\sum_{k\geq0}\mathcal{E}_{k-1}(|x_k|^2)\otimes e_k\Big\|_{Z_{E_{(2)}}^2(\M\bar{\otimes} \ell_\infty)}\lesssim_E\Big\|\sum_{k\geq0}|x_k|^2\otimes e_k\Big\|_{Z_{E_{(2)}}^2(\M\bar{\otimes} \ell_\infty)}.\end{equation}
Using quasi-triangle inequality, and combining \eqref{4a} and \eqref{4b},  we deduce that
\begin{equation}\label{4c}
\Big\|\Big(\sum_{k\geq0}|x_k|^2\Big)^{1/2}\Big\|_{E(\M)}\lesssim_E\Big\|\Big(\sum_{k\geq0}\mathcal{E}_{k-1}(|x_k|^2)\Big)^{1/2}\Big\|_{E(\M)}+\Big\|\sum_{k\geq0}|x_k|^2\otimes e_k\Big\|_{Z_{E_{(2)}}^2(\M\bar{\otimes} \ell_\infty)}^{1/2}.
\end{equation}
Let $X:=\sum_{k\geq0}|x_k|\otimes e_k$. It now remains to verify that
\begin{equation}\label{4d}
\big\|X^2\big\|_{Z_{E_{(2)}}^2(\M\bar{\otimes} \ell_\infty)}^{1/2}\lesssim_E \|X\|_{Z_E^2(\M\bar{\otimes} \ell_\infty)}.
\end{equation}
Indeed, it follows from the definition of $\|\cdot\|_{Z_{E_{(2)}}^2},$ Lemma \ref{p-concavification} and \eqref{1+2} that
\begin{eqnarray*}\big\|X^2\big\|_{Z_{E_{(2)}}^2}^{1/2}&\approx&\big\|\mu(X^2)\chi_{(0,1)}\big\|_{E_{(2)}}^{1/2}+\|X^2\|_{L_1+L_2}^{1/2}=\|\mu(X)\chi_{(0,1)}\|_E+\|X\|_{(L_1+L_2)^{(2)}}\\&\approx& \|\mu(X)\chi_{(0,1)}\|_E+\|X\|_{L_2+L_4}\leq  \|\mu(X)\chi_{(0,1)}\|_E+\|X\|_{2}
\\&\leq&2 \|\mu(X)\chi_{(0,1)}\|_E+\|\mu(X)\chi_{[1,\infty)}\|_{2}
\\&\lesssim& \|\mu(X)\chi_{(0,1)}\|_E+\|\mu(X)\|_{L_1+L_2}= \|X\|_{Z_E^2}.\end{eqnarray*}
Combining \eqref{4c} and \eqref{4d}, we complete the proof.
\end{proof}

\begin{proposition}\label{pro 2-4} Let $E\in Int(L_2(0,1),L_4(0,1))$ be a symmetric quasi-Banach space and let $x=(x_k)_{k\geq0}$ be a sequence of martingale differences. We have
\begin{eqnarray*}\|\sum_{k\geq0}x_k\|_{E(\M)}&\lesssim_{E}&\|\sum_{k\geq0}x_k\otimes e_k\|_{Z^2_E(\M\bar{\otimes}\ell_\infty)}
\\&+&\Big\|\Big(\sum_{k\geq0}\mathcal{E}_{k-1}(|x_k|^2)\Big)^{1/2}\Big\|_{E(\M)}+\Big\|\Big(\sum_{k\geq0}\mathcal{E}_{k-1}(|x_k^*|^2)\Big)^{1/2}\Big\|_{E(\M)}.
\end{eqnarray*}
\end{proposition}
\begin{proof} It follows from Theorem \ref{second main} \eqref{smb} that
$$\|\sum_kx_k\|_{E(\M)}\lesssim _{E(\M)}\Big\|\Big(\sum_{k\geq0}|x_k|^2\Big)^{1/2}\Big\|_{E(\M)}+\Big\|\Big(\sum_{k\geq0}|x_k^*|^2\Big)^{1/2}\Big\|_{E(\M)}.$$
By Proposition \ref{lemma 2-4}, we have
$$\Big\|\Big(\sum_{k\geq0}|x_k|^2\Big)^{1/2}\Big\|_{E(\M)}\lesssim_{E}\big\|\sum_{k\geq0}x_k\otimes e_k\big\|_{Z_E^2(\M\bar{\otimes} \ell_\infty)}+\Big\|\Big(\sum_{k\geq0}\mathcal{E}_{k-1}(|x_k|^2)\Big)^{1/2}\Big\|_{E(\M)}.$$
Similarly,
$$\Big\|\Big(\sum_{k\geq0}|x_k^*|^2\Big)^{1/2}\Big\|_E\lesssim_{E}\big\|\sum_{k\geq0}x_k\otimes e_k\big\|_{Z_E^2(\M\bar{\otimes} \ell_\infty)}+\Big\|\Big(\sum_{k\geq0}\mathcal{E}_{k-1}(|x_k^*|^2)\Big)^{1/2}\Big\|_{E(\M)}.$$
Combining the preceding estimates, we complete the proof.
\end{proof}

In order to prove the converse inequality in Proposition \ref{pro 2-4}, we use a crucially important result due to Junge \cite{J-D}. For convenience of the reader, we present a short proof in the appendix.

\begin{theorem}\label{junge decomposition} Let $\mathcal{M}$ be a semifinite von Neumann algebra and let $\mathcal{N}$ be a von Neumann subalgebra of $\mathcal{M}$ which admits a normal conditional expectation $\mathcal{E}:\mathcal{M}\to\mathcal{N}.$ There exists a linear isometry $u:L_2(\mathcal{M})\to L_2(\mathcal{N}\bar{\otimes}\mathcal{B}(\ell_2))$ such that
$$u(x)^*u(y)=\mathcal{E}(x^*y)\otimes e_{11},\quad \forall\, x,y\in L_2(\mathcal{M}).$$
\end{theorem}

\begin{proposition}\label{pro 2-q} Let $E\in Int(L_2(0,1),L_q(0,1)),$ $2<q<\infty,$ be a symmetric quasi-Banach function space and let $(x_k)_{k\geq0}\subset E(\M)$ be an arbitrary sequence. We have
$$\Big\|\Big(\sum_{k\geq0}\mathcal{E}_{k-1}(|x_k|^2)\Big)^{1/2}\Big\|_{E(\M)}\lesssim_{E}\Big\|\Big(\sum_{k\geq0}|x_k|^2\Big)^{1/2}\Big\|_{E(\M)}.$$
\end{proposition}
\begin{proof} We set $E(0,\infty)=\{f\in L_1(0,\infty)+L_\infty(0,\infty): \mu(f)\in E(0,1)\}$ equipped with norm $\|f\|_{E(0,\infty)}:=\|\mu(f)\|_{E(0,1)}.$ Then $(E(0,\infty),\|\cdot\|_{E(0,\infty)})$ is a quasi-Banach space. By Theorem \ref{junge decomposition}, it is easy to see for $2\leq p\leq\infty,$
$$\|u(x)\|_{L_p(\M\bar{\otimes}\mathcal{B}(\ell_2))}=\|\mathcal E(|x|^2)\|^{1/2}_{L_{\frac p2}(\M)}\leq \|x\|_{L_p(\M)},\quad x\in L_2(\M)\cap L_p(\M).$$
By density, $u$ is bounded in $L_p(\M)$ for $2\leq p\leq\infty,$ and hence $u$ is well defined on $E(\M).$
Now it follows that for an arbitrary sequence $(y_k)_{k\geq0}\subset E(\M)$
\begin{eqnarray*}\Big\|\Big(\sum_{k\geq0}\mathcal{E}_{k-1}(|y_k|^2)\Big)^{1/2}\Big\|_{E(\mathcal{M})}&=&\Big\|\Big(\sum_{k\geq0}\mathcal{E}_{k-1}(|y_k|^2)\otimes e_{11}\Big)^{1/2}\Big\|_{E(\mathcal{M}\bar{\otimes}\mathcal{B}(\ell_2))}
\\&=&\Big\|\Big(\sum_{k\geq0}u_{k-1}(y_k)^*u_{k-1}(y_k)\Big)^{1/2}\Big\|_{E(\mathcal{M}\bar{\otimes}\mathcal{B}(\ell_2))}
\\&=&\Big\|\sum_{k\geq0}u_{k-1}(y_k)\otimes e_{k1}\Big\|_{E(\mathcal{M}\bar{\otimes}\mathcal{B}(\ell_2)\bar{\otimes}\mathcal{B}(\ell_2))}.
\end{eqnarray*}
Let $2\leq r<\infty.$ Consider the mapping
$S:L_r(\mathcal{M}\bar{\otimes}\mathcal{B}(\ell_2))\to L_r(\mathcal{M}\bar{\otimes}\mathcal{B}(\ell_2)\bar{\otimes}\mathcal{B}(\ell_2))$ by setting
$$Sy=T(y\cdot (1\otimes e_{11})),\quad y=\sum_{k,j}y_{kj}\otimes e_{kj}\in L_r(\mathcal{M}\bar{\otimes}\mathcal{B}(\ell_2)),$$
where $$T:\sum_{k\geq0}y_{k1}\otimes e_{k1}\to\sum_{k\geq0}u_k(y_{k1})\otimes e_{k1}.$$
We claim that $S$ is bounded from $L_r(\mathcal{M}\bar{\otimes}\mathcal{B}(\ell_2))$ into $L_r(\mathcal{M}\bar{\otimes}\mathcal{B}(\ell_2)\bar{\otimes}\mathcal{B}(\ell_2)).$
Indeed, by the dual Doob inequality \cite[Theorem 0.1]{J-D} we have
\begin{eqnarray*}
\|Sy\|_{L_r(\mathcal{M}\bar{\otimes}\mathcal{B}(\ell_2)\bar{\otimes}\mathcal{B}(\ell_2))}&=&\|\sum_{k\geq0}u_{k-1}(y_{k1})\otimes e_{k1}\|_{L_r(\mathcal{M}\bar{\otimes}\mathcal{B}(\ell_2)\bar{\otimes}\mathcal{B}(\ell_2))}
\\&=&\Big\|\Big(\sum_{k\geq0}\mathcal{E}_{k-1}(|y_{k1}|^2)\Big)^{1/2}\Big\|_{L_r(\M)}
\\&=&\big\|\sum_{k\geq0}\mathcal{E}_{k-1}(|y_{k1}|^2)\big\|^{1/2}_{L_{\frac r2}(\M)}
\\&\leq&\big\|\sum_{k\geq0}|y_{k1}|^2\big\|^{1/2}_{L_{\frac r2}(\M)}=\Big\|\Big(\sum_{k\geq0}|y_{k1}|^2\Big)^{1/2}\Big\|_{L_r(\M)}
\\&=&\|\sum_{k\geq0}y_{k1}\otimes e_{k1}\|_{L_r(\mathcal{M}\bar{\otimes}\mathcal{B}(\ell_2))}\leq\|y\|_{L_r(\mathcal{M}\bar{\otimes}\mathcal{B}(\ell_2))}.
\end{eqnarray*}
Hence, $S$ maps boundedly from $E(\mathcal{M}\bar{\otimes}\mathcal{B}(\ell_2))$ into $E(\mathcal{M}\bar{\otimes}\mathcal{B}(\ell_2)\bar{\otimes}\mathcal{B}(\ell_2)).$ Setting $y=\sum_{k\geq0}x_k\otimes e_{k1},$ we conclude that
$$\Big\|\sum_{k\geq0}u_{k-1}(x_k)\otimes e_{k1}\Big\|_{E(\mathcal{M}\bar{\otimes}\mathcal{B}(\ell_2)\bar{\otimes}\mathcal{B}(\ell_2))}\lesssim_E\Big\|\sum_{k\geq0}x_k\otimes e_{k1}\Big\|_{E(\mathcal{M}\bar{\otimes}\mathcal{B}(\ell_2))}=\Big\|\Big(\sum_{k\geq0}|x_k|^2\Big)^{1/2}\Big\|_{E(\M)},$$ which is our desired inequality and the proof is complete.
\end{proof}

\begin{proof}[Proof of Theorem \ref{the 2-4}] It follows from Proposition \ref{pro 2-q} that
\begin{multline*}\Big\|\Big(\sum_{k\geq0}|x_k|^2\Big)^{1/2}\Big\|_{E(\M)}+\Big\|\Big(\sum_{k\geq0}|x_k^*|^2\Big)^{1/2}\Big\|_{E(\M)}\\\geq
\Big\|\Big(\sum_{k\geq0}\mathcal{E}_{k-1}(|x_k|^2)\Big)^{1/2}\Big\|_{E(\M)}+\Big\|\Big(\sum_{k\geq0}\mathcal{E}_{k-1}(|x_k^*|^2)\Big)^{1/2}\Big\|_{E(\M)}.
\end{multline*}
Applying Theorem \ref{second main} to the above estimate, we obtain
\begin{equation}\label{burk 24 1}
\|\sum_kx_k\|_{\color{red}E(\M)}\gtrsim_E\Big\|\Big(\sum_{k\geq0}\mathcal{E}_{k-1}(|x_k|^2)\Big)^{1/2}\Big\|_{E(\M)}+\Big\|\Big(\sum_{k\geq0}\mathcal{E}_{k-1}(|x_k^*|^2)\Big)^{1/2}\Big\|_{E(\M)}.
\end{equation}
By Theorem \ref{first main}, we have
\begin{equation}\label{burk 24 2}
\|\sum_kx_k\|_E\gtrsim_E\big\|\sum_{k\geq0}x_k\otimes e_k\big\|_{Z_E^2(\M\bar{\otimes} \ell_\infty)}.
\end{equation}
The assertion follows now by combining Proposition \ref{pro 2-4}, \eqref{burk 24 1} and \eqref{burk 24 2}.
\end{proof}

As noted in the introduction, Theorem \ref{the 2-4} partly resolves one question stated in \cite{NW}. At the time of writing this paper, we do not know how to extend Theorem \ref{the 2-4} to the case that $E\in Int(L_2, L_q) $ for some $4<q<\infty.$ However, we have the following substitute, given in Theorem \ref{burk conditional} below. We start with a technical lemma.

\begin{lemma}\label{jsz-like lemma} Let $\Phi$ be a $2$-convex and $q$-concave Orlicz function such that
\begin{equation}\label{phi-c}\lim_{t\to0}\frac{\Phi(t)}{t^2}=0,\quad\lim_{t\to\infty}\frac{\Phi(t)}{t^q}=0,\quad 2<q<\infty.
\end{equation}
Suppose that $\mathcal{M}$ be a semifinite infinite von Neumann algebra and $x\in L_\Phi(\M).$
Then there exists a measurable function $x_{\Phi}$ on $(0,\infty)$ such that
$$\|x\|_{L_\Phi(\M)}\approx_q\|x\otimes x_{\Phi}\|_{(L_2+L_q)(\M\otimes L_\infty(0,\infty))}.$$
\end{lemma}
\begin{proof} Let $\Phi$ satisfy the assumptions. By \cite[Lemma 6]{AS2014}, the function $\phi$ defined by the setting $\Phi(t)=t^2\phi(t^{q-2})$ is quasi-concave. By \cite[Theorem II.1.1]{KPS}, there exists a concave increasing function $\phi_0$ such that $\frac12\phi_0\leq\phi\leq\phi_0.$ It follows from \eqref{phi-c} that $\phi_0(0)=0$ and $\lim_{t\to\infty}\frac{\phi_0(t)}{t}=0.$ By \cite[Lemma 5.4.3]{BergLofstrom}, we have
$$\phi_0(t)=\int_0^{\infty}\min\{t,\tau\}d(-\phi_0'(\tau)).$$
Therefore,
\begin{eqnarray*}\Phi(t)&\approx& \int_0^{\infty}\min\{t^q,\tau t^2\}d(-\phi_0'(\tau))\stackrel{\tau=s^{2-q}}{=}\int_0^{\infty}\min\{t^q,s^{2-q}t^2\}d(\phi_0'(s^{2-q}))
\\&=&\int_0^{\infty}\min\{(ts)^q,(ts)^2\}d\nu(s)\approx_q\int_0^{\infty}N_q(ts)d\nu(s)\stackrel{def}{=}\Phi_0(t),
\end{eqnarray*}
where the positive measure $d\nu(s)=s^{-q}d(\phi_0'(s^{2-q}))$ and the Orlicz function $N_q$ is defined by the setting
$$N_q(t)=
\begin{cases}
2t^q,\quad 0\leq t\leq 1\\
qt^2-q+2,\quad t\geq1.
\end{cases}
$$
The equivalence $\Phi(t)\approx_q \Phi_0(t)$ for all $t>0$ guarantees that $L_\Phi(\M)=L_{\Phi_0}(\M).$ Now let $\|x\|_{L_{\Phi_0}(\M)}=\lambda.$ It follows from $\mu(x\otimes y)=\mu(\mu(x)\otimes \mu(y))$ that
\begin{eqnarray*}
1=\int_0^\infty\Phi_0\Big(\frac{\mu(t,x)}{\lambda}\Big)dt&=&\int_0^\infty\int_0^\infty N_q\Big(\frac{\mu(t,x)s}{\lambda}\Big)d\nu(s)dt\\&=&(\tau\otimes\int)\Big(N_q\Big(\frac {x\otimes x_\Phi}{\lambda}\Big)\Big),
\end{eqnarray*}
where $x_{\Phi}$ is a measurable function on $(0,\infty)$ whose distribution is $\nu.$ Now the well-known property of Orlicz norms (see e.g. \cite[I,1.2,Proposition 11]{MM}) implies that
$$\|x\|_{L_\Phi(\M)}\approx_q\|x\|_{L_{\Phi_0}(\M)}=\|x\otimes x_\Phi\|_{L_{N_q}(\M\otimes L_\infty(0,\infty))}.$$
One can easily show that $(L_2+L_q)(0,\infty)=L_{N_q}(0,\infty)$, and we conclude this proof.
\end{proof}

The following assertion is proved in \cite[Theorem 7.1]{KM} (see also \cite[Theorem 2.7]{ASW2}).

\begin{theorem}\label{km theorem} Let $E\in Int(L_p,L_q),$ $1\leq p<q\leq\infty,$ be a symmetric Banach function space. If $z$ and $w$ are such that $\|z\|_{L_{\Phi}}\leq\|w\|_{L_{\Phi}}$ for every $p$-convex and $q$-concave Orlicz function $\Phi$ and $w\in E,$ then $z\in E$ and $\|z\|_E\lesssim_E\|w\|_E.$
\end{theorem}
The following result shows that the remaining part concerning the Burkholder inequality for symmetric Banach spaces $E\in Int(L_2,L_q)$ for $4<q<\infty$ needs to be answered only for $E=L_2+L_q.$
\begin{theorem}\label{burk conditional} Let $\mathcal{M}$ be a semifinite infinite von Neumann algebra. Suppose that for every sequence $(x_k)_{k\geq0}\subset (L_2+L_q)(\mathcal{M})$ of martingale differences, $2<q<\infty,$ the inequality
\begin{eqnarray*}\|\sum_{k\geq0}x_k\|_{(L_2+L_q)(\mathcal{M})}&\lesssim_q&\Big\|\sum_{k\geq0}x_k\otimes e_k\Big\|_{(L_2+L_q)(\M\bar{\otimes}\ell_\infty)}\\&+&\Big\|\Big(\sum_{k\geq0}\mathcal{E}_{k-1}(|x_k|^2)\Big)^{1/2}\Big\|_{(L_2+L_q)(\mathcal{M})}
\\&+&\Big\|\Big(\sum_{k\geq0}\mathcal{E}_{k-1}(|x_k^*|^2)\Big)^{1/2}\Big\|_{(L_2+L_q)(\mathcal{M})}
\end{eqnarray*}
holds. If $E\in Int(L_2,L_q)$ is a symmetric Banach function space on the semi-axis, then for every sequence $(x_k)_{k\geq0}\subset E(\mathcal{M})$ of martingale differences the inequality
\begin{eqnarray*}\|\sum_{k\geq0}x_k\|_{E(\M)}&\lesssim_E&\Big\|\sum_{k\geq0}x_k\otimes e_k\Big\|_{E(\M\bar{\otimes}\ell_\infty)}+
\\&+&\Big\|\Big(\sum_{k\geq0}\mathcal{E}_{k-1}(|x_k|^2)\Big)^{1/2}\Big\|_{E(\M)}+\Big\|\Big(\sum_{k\geq0}\mathcal{E}_{k-1}(|x_k^*|^2)\Big)^{1/2}\Big\|_{E(\M)}
\end{eqnarray*}
holds.
\end{theorem}
\begin{proof} Let $(x_k)_{k\geq0}\subset E(\mathcal{M})$ be an arbitrary sequence of martingale differences. Let $\Phi$ be as in Lemma \ref{jsz-like lemma}. Take $x_{\Phi}$ on $(0,\infty)$ from Lemma \ref{jsz-like lemma}. Consider a sequence $(x_k\otimes x_{\Phi})_{k\geq0}.$ This sequence consists of martingale differences with respect to the filtration $(\mathcal{M}_k\otimes L_{\infty}(0,\infty))_{k\geq0}.$ By Lemma \ref{jsz-like lemma} we have
\begin{equation}\label{from lemma 8}\|\sum_{k\geq0}x_k\|_{L_{\Phi}(\M)}\approx_q\|\sum_{k\geq0}x_k\otimes x_{\Phi}\|_{(L_2+L_q)(\M\bar{\otimes}L_\infty(0,\infty))}.
\end{equation}
Applying the assumption to the sequence $(x_k\otimes x_{\Phi})_{k\geq0}$ and combining \eqref{from lemma 8} we obtain
\begin{eqnarray*}\|\sum_{k\geq0}x_k\|_{L_{\Phi}(\M)}&\lesssim_q&\|\sum_{k\geq0}(x_k\otimes x_{\Phi})\otimes e_k\|_{(L_2+L_q)(\M\bar{\otimes}L_\infty(0,\infty)\bar{\otimes}\ell_\infty)}
\\&+&\Big\|\Big(\sum_{k\geq0}(\mathcal{E}_{k-1}\otimes 1)(|x_k|^2\otimes x_{\Phi}^2)\Big)^{1/2}\Big\|_{(L_2+L_q)(\M\bar{\otimes}L_\infty(0,\infty))}\\&+&\Big\|\Big(\sum_{k\geq0}(\mathcal{E}_{k-1}\otimes 1)(|x_k^*|^2\otimes x_{\Phi}^2)\Big)^{1/2}\Big\|_{(L_2+L_q)(\M\bar{\otimes}L_\infty(0,\infty))}
\\&=&\|(\sum_{k\geq0}x_k\otimes e_k)\otimes x_{\Phi}\|_{(L_2+L_q)(\M\bar{\otimes} \ell_\infty\bar{\otimes}L_\infty(0,\infty))} \\&+&\|\Big(\sum_{k\geq0}\mathcal{E}_{k-1}(|x_k|^2)\Big)^{1/2}\otimes x_{\Phi}\|_{(L_2+L_q)(\M\bar{\otimes}L_\infty(0,\infty))}\\&+&\|\Big(\sum_{k\geq0}\mathcal{E}_{k-1}(|x_k^*|^2)\Big)^{1/2}\otimes x_{\Phi}\|_{(L_2+L_q)(\M\bar{\otimes}L_\infty(0,\infty))}
\\&\approx_q&\|\sum_{k\geq0}x_k\otimes e_k\|_{L_{\Phi}(\M\bar{\otimes} \ell_\infty)}+\|\Big(\sum_{k\geq0}\mathcal{E}_{k-1}(|x_k|^2)\Big)^{1/2}\|_{L_{\Phi}(\M)}\\&+&\|\Big(\sum_{k\geq0}\mathcal{E}_{k-1}(|x_k^*|^2)\Big)^{1/2}\|_{L_{\Phi}(\M)}.
\end{eqnarray*}
The assertion follows now from Theorem \ref{km theorem}.
\end{proof}

\section{Disjointification inequalities: the $\Phi$-moment case}

In this section, we prove $\Phi$-moment version of {\color{red}disjointification} inequalities for noncommutative martingale difference sequences. The main result is Theorem \ref{third main}, which is new even in the classical case.

We consider the Orlicz function
\begin{equation}\label{M}
M(t)=
\begin{cases}
t^2,\quad 0\leq t\leq 1\\
2t-1,\quad t\geq1.
\end{cases}
\end{equation}

\begin{lemma}\label{m simple est} Let $(\M,\tau)$ and $(\N,\sigma)$ be finite von Neumann algebras. If $T:L_1(\N)\to L_1(\M)$ and $T:L_2(\N)\to L_2(\M)$ is a contraction, then
$$\tau(M(|Tx|))\leq 4\sigma(M(|x|)).$$
\end{lemma}
\begin{proof} Split $x=x_1+x_2,$ where $x_1=xe_{[0,1]}(|x|)$ and $x_2=xe_{(1,\infty)}(|x|).$ By \cite[Theorem 3.3.3]{LSZ}, we have
$${\color{red}\mu(Tx)=\mu(Tx_1+Tx_2)}\prec\prec\mu(Tx_1)+\mu(Tx_2)=\frac12\Big(\mu(2Tx_1)+\mu(2Tx_2)\Big).$$
and, therefore,
$$\tau(M(|Tx|))\leq{\color{red}\mathbb E}\Big(M\Big(\frac12\big(\mu(2Tx_1)+\mu(2Tx_2)\big)\Big)\Big)\leq\frac12\Big(\tau(M(2|Tx_1|))+
\tau(M(2|Tx_2|))\Big).$$
It is easy to see $M(2t)\leq 4M(t),$ $M(t)\leq 2t$ and $M(t)\leq t^2$ for every $t>0,$ and hence we get
$$\tau(M(|Tx|))\leq 2\tau(M(|Tx_1|))+2\tau(M(|Tx_2|))\leq 2\tau(|Tx_1|^2)+4\tau(|Tx_2|).$$
Since $T:L_1(\N)\to L_1(\M)$ and $T:L_2(\N)\to L_2(\M)$ is a contraction, it follows that
$$\tau(M(|Tx|))\leq 2\sigma(|x_1|^2)+4\sigma(|x_2|)\leq 4{\color{red}\sigma}(M(|x_1|))+4{\color{red}\sigma}(M(|x_2|))=4\sigma(M(|x|)).$$
\end{proof}

\begin{lemma}\label{modular left m} Let $(\mathcal{M},\tau)$ be a finite von Neumann algebra and let $(\mathcal{N},\nu)$ be an atomless semifinite one. If $T:L_2(\N)\to L_2(\M)$ and $T:L_1(\N)\to L_1(\M)$ is a linear contraction, then for $M$ given by \eqref{M} we have
$$\tau(M(|Tx|))\lesssim\Big(\int_0^1M(\mu(s,x))ds+M(\|x\|_{L_1+L_2})\Big).$$
\end{lemma}
\begin{proof} To see this, fix $x=x^*\in L_0(\N)$ and consider the spectral projection
$$p=e_{(\mu(1,x),\infty)}(|x|).$$
Clearly, trace of $p$ is less than or equal to $1.$ Fix a projection $q\geq p$ whose trace equals $1.$ Clearly, $T:L_1(q\N q)\to L_1(\M)$ and $T:L_2(q\N q)\to L_2(\M)$ is a contraction.
Again by \cite[Theorem 3.3.3]{LSZ}, we have
$${\color{red}\mu(Tx)}\prec\prec\mu(T(pxp))+\mu(T((1-p)xp))+\mu(T(x(1-p))).$$
Therefore,
\begin{eqnarray*}\tau(M(|Tx|))&\leq&{\color{red}\mathbb E}\Big(M\Big(\mu(T(pxp))+\mu(T((1-p)xp))+\mu(T(x(1-p)))\Big)
\\&\lesssim&\tau(M(T(3|pxp|)))+\tau(M(3|T((1-p)xp)|))+\tau(M(3|T(x(1-p))|)).
\end{eqnarray*}
Applying Lemma \ref{m simple est} to the finite algebras $\mathcal{M}$ and $q\mathcal{N}q$ and noting that $M(3t)\leq 9M(t)$ for every $t>0$, we obtain that
$$\tau(M(T(3|pxp|)))\leq 4\sigma(M(3|pxp|))\leq 36\sigma(M(|pxp|))\leq 36\int_0^1M(\mu(s,x))ds.$$
Define concave function $M_0$ by setting $M_0(t)=M(t^{1/2}),\,t>0.$ It follows from Jensen inequality and the assumption that
$$\tau(M(|Tz|))=\tau(M_0(|Tz|^2))\leq M_0(\tau(|Tz|^2))=M_0(\|Tz\|_2^2)=M(\|Tz\|_2)\leq M(\|z\|_2).$$
Therefore,
\begin{multline*}\tau(M(3|T((1-p)xp)|))+\tau(M(3|T(x(1-p))|))
\leq 2M(3\|x(1-p)\|_2)\\=2M(3\|\mu(x)\chi_{(1,\infty)}\|_2)\lesssim 2M(3\|x\|_{L_1+L_2})\leq 18M(\|x\|_{L_1+L_2}).
\end{multline*}
\end{proof}

\begin{lemma}\label{modular left sum interpolation} Let $(\mathcal{M},\tau)$ be a finite von Neumann algebra and let $(\mathcal{N},\nu)$ be a semifinite atomless one. Let $\Phi$ be a $2$-concave Orlicz function. If $T:L_2(\N)\to L_2(\M)$ and $T:L_1(\N)\to L_1(\M)$ is a linear contraction, then
$$\tau(\Phi(|Tx|))\lesssim\int_0^1\Phi(\mu(s,x))ds+\Phi(\|x\|_{L_1+L_2}).$$
\end{lemma}
\begin{proof} Define the function $\phi:(0,\infty)\to(0,\infty)$ by setting $\Phi(t)=t\phi(t),$ $t>0.$ Since $\Phi$ is convex and $2$-concave, it follows that $\phi$ is quasi-concave; see \cite[Lemma 6]{AS2014}. By \cite[Theorem II.1.1] {KPS}, there exists a concave increasing function $\phi_0$ such that $\frac12\phi_0\leq\phi\leq\phi_0.$ Without loss of generality, $\phi_0(0)=0$ and $\phi_0(t)=o(t),$ $t\to\infty.$ By \cite[Lemma 5.4.3]{BergLofstrom}, we have
$$\phi_0(t)=\int_0^{\infty}\min\{t,\lambda\}d(-\phi_0'(\lambda)).$$
Arguing as in Lemma \ref{jsz-like lemma}, we obtain that
$$\Phi(t)\approx\int_0^{\infty}M(\frac{t}{\lambda})d\nu(\lambda)$$
for some (positive) measure $\nu$ on $(0,\infty).$
Hence,
$$\tau(\Phi(|x|))\approx\int_0^{\infty}\tau\Big(M(\frac{|x|}{\lambda})\Big)d\nu(\lambda).$$
Therefore, by Lemma \ref{modular left m}, we have
\begin{eqnarray*}\tau(\Phi(|Tx|))&\approx&\int_0^{\infty}\tau\Big(M(\frac{|Tx|}{\lambda})\Big)d\nu(\lambda)
\\&\lesssim&\int_0^\infty
\Big(\int_0^1M(\mu(s,\frac{x}{\lambda}))ds+M(\frac1\lambda\|x\|_{L_1+L_2})\Big)d\nu(\lambda)
\\&=&\int_0^1\int_0^\infty M\Big(\mu(s,\frac{x}{\lambda})\Big)d\nu(\lambda)ds+\int_0^\infty M\Big(\frac1\lambda\|x\|_{L_1+L_2}\Big)d\nu(\lambda)
\\&\approx&\int_0^1\Phi(\mu(s,x))ds+\Phi(\|x\|_{L_1+L_2}),
\end{eqnarray*} which is as our desired inequality and the proof is complete.
\end{proof}

\begin{lemma}\label{right c head} Let $\Phi$ be a $2$-convex Orlicz function. We have
$$(\tau\otimes\int)\Big(\Phi\Big(|\sum_{k\geq0}y_k\otimes r_k|\Big)\Big)\geq\tau(\Phi((\sum_{k\geq0}|y_k|^2)^{1/2})).$$
\end{lemma}
\begin{proof} Without loss of generality, we have finitely many (let us say, $n$) summands in the inequality above. Let
$\mathcal{R}_n$ be the set of all maps from $\{0,\cdots,n-1\}$ into $\{-1,1\},$ that is
$$\mathcal{R}_n=\Big\{\epsilon:\{0,\cdots,n-1\}\to\{-1,1\}\Big\}.$$
For every $\epsilon\in\mathcal{R}_n,$ we set
$$T_{\epsilon}:=|\sum_{k=0}^{n-1}\epsilon(k)y_k|^2\quad {\rm and}\quad A_\epsilon:=\big\{t\in (0,1):r_k(t)=\epsilon(k),0\leq k\leq n-1\big\}.$$
It is clear that the Lebesgue measure of $A_\epsilon$ is $2^{-n}.$
Set $\Phi_0(t)=\Phi(t^{1/2}),$ $t>0.$ Since $\Phi$ is $2$-convex, it follows that $\Phi_0$ is convex. Hence, we have
\begin{equation}\label{8}
(\tau\otimes\int)\Big(\Phi\Big(|\sum_{k=0}^{n-1}y_k\otimes r_k|\Big)\Big)=\sum_{\epsilon\in\mathcal{R}_n}(\tau\otimes\int_{A_\epsilon})\Big(\Phi\Big(|\sum_{k=0}^{n-1}y_k\otimes r_k|\Big)\Big)=2^{-n}\sum_{\epsilon\in\mathcal{R}_n}\tau(\Phi_0(T_{\epsilon})).
\end{equation}
Applying \eqref{convex phi}, we obtain
\begin{equation}\label{9}2^{-n}\sum_{\epsilon\in\mathcal{R}_n}\tau(\Phi_0(T_{\epsilon}))\geq\tau(\Phi_0(2^{-n}\sum_{\epsilon\in\mathcal{R}_n}T_{\varepsilon})).
\end{equation}
Note that for $0\leq l,m\leq n-1,$
$$\sum_{\epsilon\in \mathcal{R}_n}\epsilon(l)\epsilon(m)=
\begin{cases}
2^{n},\quad l=m\\
0,\quad l\neq m.
\end{cases}$$
It is immediate that
$$2^{-n}\sum_{\epsilon\in\mathcal{R}_n}T_{\varepsilon}=2^{-n}\sum_{\epsilon\in\mathcal{R}_n}|\sum_{k=0}^{n-1}\epsilon(k)y_k|^2=2^{-n}\sum_{l,m=0}^{n-1}y_l^*y_m\sum_{\epsilon\in \mathcal{R}_n}\epsilon(l)\epsilon(m)=\sum_{k=0}^{n-1}|y_k|^2.$$
Consequently,
\begin{equation}\label{99}\tau(\Phi_0(2^{-n}\sum_{\epsilon\in\mathcal{R}_n}T_{\varepsilon}))=\tau(\Phi_0(\sum_{k=0}^{n-1}|y_k|^2))=\tau(\Phi((\sum_{k=0}^{n-1}|y_k|^2)^{1/2})).
\end{equation}
Combining \eqref{8}, \eqref{9} and \eqref{99}, we conclude the proof.
\end{proof}

\begin{lemma}\label{right c tail} Let $\Phi$ be an arbitrary Orlicz function. If $0\leq z_k\in L_\Phi(\M),$ $k\geq0,$ then we have
$$\tau(\Phi(\sum_{k\geq0}z_k))\geq\frac12\Big({\color{red}\mathbb E}(\Phi(\mu(Z)\chi_{(0,1)}))+\Phi(\|Z\|_1)\Big),\quad Z=\sum_{k\geq0}z_k\otimes e_k.$$
\end{lemma}
\begin{proof} Without loss of generality, suppose that the spectral projection $e_t({z_k})=0$ for every $t>0.$ Set
$$z_{1k}=z_ke_{(\mu(1,Z),\infty)}(z_k),\quad z_{2k}=z_ke_{[0,\mu(1,Z)]}(z_k).$$
By \cite[Lemma 3.3.7]{LSZ}, we have
$$\sum_{k\geq0}z_{1k}\otimes e_k\prec\prec\sum_{k\geq0}z_{1k}$$
and, therefore,
$$\mu(Z)\chi_{(0,1)}\prec\prec{\color{red}\mu\Big(\sum_{k\geq0}z_k\Big)}.$$
Thus, by \eqref{majorization phi}
$$\tau(\Phi(\sum_{k\geq0}z_k))\geq {\color{red}\mathbb E}(\Phi(\mu(Z)\chi_{(0,1)})).$$
By Jensen inequality, we have
$$\tau(\Phi(\sum_{k\geq0}z_k))=\int_0^1\Phi(\mu(t,\sum_{k\geq0}z_k))dt\geq\Phi(\tau(\sum_{k\geq0}z_k))=\Phi(\sum_{k\geq0}\|z_k\|_1)=\Phi(\|Z\|_1).$$
The assertion follows by combining the latter two inequalities.
\end{proof}

Now we are in a position to prove our third main result, Theorem \ref{third main}.
\begin{proof}[Proof of Theorem \ref{third main}] We first prove the first assertion.
Consider the operator $T$ defined as the proof of Proposition \ref{left estimate}. Then $T$ is both bounded from $L_1(\M\bar{\otimes}\ell_\infty)$ to $L_1(\M)$ and from $L_2(\M\bar{\otimes}\ell_\infty)$ to $L_2(\M).$
Suppose now $\{x_k\}_{k\geq0}$ is a sequence of martingale differences. Applying  Lemma \ref{modular left sum interpolation}, we have
\begin{eqnarray*}\tau(\Phi(|\sum_{k\geq0}x_k|))&=&\tau(\Phi(|TX|))\lesssim\int_0^1\Phi\Big(\mu(s,X)\Big)ds+\Phi(\|X\|_{L_1+L_2}).
\end{eqnarray*}
We now turn to the proof of the second assertion. Since $\Phi$ is $2$-convex and $q$-concave, it follows from \cite[Corollary 3.1]{BC} that, for every $t\in(0,1)$
$$\tau\Big(\Phi\Big(|\sum_{k\geq0}x_k|\Big)\Big)\approx_{{\Phi}}\tau\Big(\Phi\Big(|\sum_{k\geq0}r_k(t)x_k|\Big)\Big).$$
Hence,
$$\tau\Big(\Phi\Big(|\sum_{k\geq0}x_k|\Big)\Big)\approx_{{\Phi}}(\tau\otimes\int)\Big(\Phi\Big(|\sum_{k\geq0}r_k\otimes x_k|\Big)\Big).$$
Set $\Phi_0(t):=\Phi(t^{1/2}),$ $t>0.$ It follows from Lemma \ref{right c head} that
\begin{equation}\label{6.1}
\tau\Big(\Phi\Big(|\sum_{k\geq0}x_k|\Big)\Big)\gtrsim_\Phi\tau\Big(\Phi\Big((\sum_{k\geq0}|x_k|^2)^{1/2}\Big)\Big)=\tau\Big(\Phi_0\Big(\sum_{k\geq0}|x_k|^2\Big)\Big).
\end{equation}
Since $\Phi$ is $2$-convex, it follows that $\Phi_0$ is convex. By Lemma \ref{right c tail}, we have
\begin{eqnarray*}\tau\Big(\Phi_0\Big(\sum_{k\geq0}|x_k|^2\Big)\Big)&\geq&
\frac12\Big({\color{red}\mathbb E}(\Phi_0\Big(\mu(\sum_{k\geq 0}|x_k|^2\otimes e_k)\chi_{(0,1)})\Big)+\Phi_0(\|\sum_{k\geq 0}|x_k|^2\otimes e_k\|_1)\Big)
\\&\geq&\frac12\Big({\color{red}\mathbb E}(\Phi_0\Big(\mu^2(X)\chi_{(0,1)})\Big)+\Phi_0(\|X\|_2^2)\Big)
\\&=&\frac12{\color{red}\mathbb E}(\Phi(\mu(X)\chi_{(0,1)}))+\Phi(\|X\|_2).
\end{eqnarray*}
Combining this inequality and \eqref{6.1}, we conclude the proof.
\end{proof}

Now we show that the assumption on $\Phi$ in Theorem \ref{third main} (\ref{tmb}) is sharp in the sense that the condition of $q$-concavity is necessary.
\begin{proposition}\label{q-concavity necessary} Let $(\M,\tau)$ be an arbitrary noncommutative probability space. Suppose that an Orlicz function $\Phi$ is such that for every sequence of martingale differences $(x_k)_{k\geq 0}\subset L_\Phi(\M)$ we have
$${\color{red}\mathbb E}\Big(\Phi(\mu(X)\chi_{(0,1)})\Big)\lesssim_{\Phi}\tau\Big(\Phi\Big(\sum_{k\geq0}x_k\Big)\Big),\quad X=\sum_{k\geq0}x_k\otimes e_k.$$
It follows that $\Phi$ is $q$-concave for some $q<\infty.$
\end{proposition}
\begin{proof} Let $\mathcal{M}_0=\mathbb{C}$ and take $x\in\mathcal{M}_1.$ That is, our sequence $\{x_k\}_{k\geq0}$ is defined by the setting $x_0=\tau(x),$ $x_1=x-\tau(x)$ and $x_k=0,$ $k>1.$ {\color{red} Let $t\in \mathbb {R}_+$ be an arbitrary fixed constant. Take 
$$x=9t\chi_{(0,\frac13)}-6t\chi_{(\frac13,1)}.$$
We have
$$x_0=-t,\quad x_1=10t\chi_{(0,\frac13)}-5t\chi_{(\frac13,1)}.$$}
Hence,
$$\mu(X)=10t\chi_{(0,\frac13)}+5t\chi_{(\frac13,1)}+t\chi_{(1,2)}.$$
Therefore, by the assumption, we have
$$\frac13\Phi(10t)+\frac23\Phi(5t)\lesssim_{\Phi}\frac13\Phi(9t)+\frac23\Phi(6t).$$
Hence,
$$\Phi(10t)\lesssim_{\Phi}\Phi(9t).$$
This implies that $\Phi$ satisfies $\Delta_2$-condition, and hence $\Phi$ is $q$-concave. Indeed, if $\Phi$ satisfies $\Delta_2$-condition, then it is easy to find $0<q<\infty$ such that $\frac{\Phi(t)}{t^q}, t>0$ is decreasing. Since $\Phi$ is convex, $\frac{\Phi(t)}{t}, t>0$ is increasing. By \cite[Lemma 8]{JSZ}, we have that $\Phi$ is $q$-concave.

\end{proof}

\section{Noncommutative martingale inequalities: the $\Phi$-moment case }
In this section we prove the $\Phi$-moment Burkholder-Gundy inequalities by using results established in the preceding section.
\begin{lemma}\label{7-lemma}Let $\Phi$ be an Orlicz function and let $(x_k)_{k\geq0}$ be an arbitrary sequence from $L_\Phi(\M)$. If $\Phi$ is an $2$-concave Orlicz function, then
$$(\tau\otimes\int)\Big(\Phi\Big(|\sum_{k\geq0}x_k\otimes r_k|\Big)\Big)\lesssim\tau(\Phi((\sum_{k\geq0}|x_k|^2)^{1/2})).$$
\end{lemma}

\begin{proof} Consider the operator $S$ defined in Lemma \ref{left hinchine estimate}. Observe that the operator $S$ satisfies the assumptions in Lemma \ref{modular left sum interpolation}. Applying this lemma, we have
$$(\tau\otimes\int)(\Phi(|Sx|))\lesssim\int_0^1\Phi(\mu(s,x))ds+\Phi(\|x\|_{L_1+L_2}),\quad x\in L_0(\M\bar{\otimes}B(\ell_2)).$$
Setting $x=\sum_{k\geq0}x_k\otimes e_{k1},$ we write
$$(\tau\otimes\int)\Big(\Phi\Big(|\sum_{k\geq0}x_k\otimes r_k|\Big)\Big)\lesssim\int_0^1\Phi(\mu(s,x))ds+\Phi(\|x\|_{L_1+L_2}).$$
However, $$\mu\Big(s,\sum_{k\geq0}x_k\otimes e_{k1}\Big)=\mu\Big(s,(\sum_{k\geq0}|x_k|^2)^{1/2}\Big),\quad s>0$$
which lives on $(0,1)$, and therefore,
\begin{eqnarray*}(\tau\otimes\int)\Big(\Phi\Big(|\sum_{k\geq0}x_k\otimes r_k|\Big)\Big)&\lesssim&\tau(\Phi((\sum_{k\geq0}|x_k|^2)^{1/2}))+\Phi\Big(\int_0^1\mu\Big(s,(\sum_{k\geq0}|x_k|^2)^{1/2}\Big)ds\Big)
\\&\leq&\tau(\Phi((\sum_{k\geq0}|x_k|^2)^{1/2}))+\int_0^1\Phi\Big(\mu\Big(s,(\sum_{k\geq0}|x_k|^2)^{1/2}\Big)\Big)ds
\\&=&2\tau(\Phi((\sum_{k\geq0}|x_k|^2)^{1/2})).
\end{eqnarray*}
This proof is complete.
\end{proof}
 The following result is the main result of this section. It should be compared with \cite[Theorem 5.1]{BC}.
\begin{theorem}\label{fourth main} Let $\Phi$ be an Orlicz function and let $(x_k)_{k\geq0}\in L_\Phi(\M)$ be a sequence of martingale differences.
\begin{enumerate}[{\rm (i)}]
\item\label{4ma} If $\Phi$ is $p$-convex for some $1<p<2$ and $2$-concave, then
$$\tau\Big(\Phi\Big(|\sum_{k\geq0}x_k|\Big)\Big)\approx_{\Phi}\inf_{x_k=y_k+z_k}\Big\{\tau\Big(\Phi\Big[\Big(\sum_{k\geq0}|y_k|^2\Big)^{1/2}\Big]\Big)+\tau\Big(\Phi\Big[\Big(\sum_{k\geq0}|z_k^*|^2\Big)^{1/2}\Big]\Big)\Big\},$$
where the infimum is taken over the sequences of martingale differences.
\item\label{4mb} If $\Phi$ is $2$-convex and $q$-concave for some $2<q<\infty,$ then
$$\tau\Big(\Phi\Big(|\sum_{k\geq0}x_k|\Big)\Big)\approx_{\Phi}\max\Big\{\tau\Big(\Phi\Big[\Big(\sum_{k\geq0}|x_k|^2\Big)^{1/2}\Big]\Big),\tau\Big(\Phi\Big[\Big(\sum_{k\geq0}|x_k^*|^2\Big)^{1/2}\Big]\Big)\Big\}.$$
\end{enumerate}
\end{theorem}
\begin{proof} Since $\Phi$ is $p$-convex and $q$-concave for $1<p<q<\infty,$ it follows from \cite[Corollary 3.1]{BC} that
\begin{equation}\label{arbitrary to tensor1}
\tau\Big(\Phi\Big(|\sum_{k\geq0}x_k|\Big)\Big)\approx_{\Phi}(\tau\otimes \int)\Big(\Phi\Big(|\sum_{k\geq0}x_k\otimes r_k|\Big)\Big).
\end{equation}

\eqref{4ma} It follows from \eqref{arbitrary to tensor1} and the proof of \cite[Theorem 4.1 (i)]{BC}\footnote{Strictly speaking, the argument of \cite[Theorem 4.1 (i)]{BC} works well for every $p$-convex and $q$-concave Orlicz function with $1<p<q<\infty$ without any change.} that
$$\tau\Big(\Phi\Big(|\sum_{k\geq0}x_k|\Big)\Big)\gtrsim_{\Phi}\inf_{x_k=y_k+z_k}\Big\{\tau\Big(\Phi\Big[\Big(\sum_{k\geq0}|y_k|^2\Big)^{1/2}\Big]\Big)+\tau\Big(\Phi\Big[\Big(\sum_{k\geq0}|z_k^*|^2\Big)^{1/2}\Big]\Big)\Big\},$$
where the infimum is taken over all possible decompositions $x_k=y_k+z_k$ with $y_k$ and $z_k$ in $L_\Phi(\M)$.
For any fixed decomposition $x_k=y_k+z_k$, since $(x_k)_{k\geq 0}$ is a martingale difference sequence,
$$x_k=(\mathcal E_k(y_k)-\mathcal E_{k-1}(y_k))+ (\mathcal E_k(y_k)-\mathcal E_{k-1}(y_k)):=u_k+v_k.$$ Then $(u_k)_{k\geq 0}$ and $(v_k)_{k\geq 0}$ are martingale difference sequences, and by the Stein inequality \cite[Theorem 3.2]{BC} and \eqref{quasi-trangle}, we have
$$\tau\Big(\Phi\Big[\Big(\sum_{k\geq0}|u_k|^2\Big)^{1/2}\Big]\Big)\lesssim_\Phi \tau\Big(\Phi\Big[\Big(\sum_{k\geq0}|y_k|^2\Big)^{1/2}\Big]\Big)$$
and
$$\tau\Big(\Phi\Big[\Big(\sum_{k\geq0}|v_k^*|^2\Big)^{1/2}\Big]\Big)\lesssim_\Phi \tau\Big(\Phi\Big[\Big(\sum_{k\geq0}|z_k^*|^2\Big)^{1/2}\Big]\Big).$$
Hence, we have
$$\tau\Big(\Phi\Big(|\sum_{k\geq0}x_k|\Big)\Big)\gtrsim_{\Phi}\inf_{x_k=y_k+z_k}\Big\{\tau\Big(\Phi\Big[\Big(\sum_{k\geq0}|y_k|^2\Big)^{1/2}\Big]\Big)+\tau\Big(\Phi\Big[\Big(\sum_{k\geq0}|z_k^*|^2\Big)^{1/2}\Big]\Big)\Big\},$$
where the infimum is taken over the sequences of martingale differences.

We now prove the converse inequality. Let $x_k=y_k+z_k,$ with $(y_k)_{k\geq0}\subset E(\M)$ and $(z_k)_{k\geq0}\subset E(\M)$ being sequences of martingale differences. Applying Lemma \ref{7-lemma} and \eqref{arbitrary to tensor1}, we conclude that
$$\tau\Big(\Phi\Big(|\sum_{k\geq0}y_k|\Big)\Big)\lesssim_{\Phi}\tau\Big(\Phi\Big[\Big(\sum_{k\geq0}|y_k|^2\Big)^{1/2}\Big]\Big)$$
and, similarly,
$$\tau\Big(\Phi\Big(|\sum_{k\geq0}z_k|\Big)\Big)\lesssim_{\Phi}\tau\Big(\Phi\Big[\Big(\sum_{k\geq0}|z_k^*|^2\Big)^{1/2}\Big]\Big).$$
Therefore, again by \eqref{quasi-trangle} we have
$$\tau\Big(\Phi\Big(|\sum_{k\geq0}x_k|\Big)\Big)\lesssim_{\Phi}\inf_{x_k=y_k+z_k}\Big\{\tau\Big(\Phi\Big[\Big(\sum_{k\geq0}|y_k|^2\Big)^{1/2}\Big]\Big)+\tau\Big(\Phi\Big[\Big(\sum_{k\geq0}|z_k^*|^2\Big)^{1/2}\Big]\Big)\Big\},$$
where the infimum is taken over the sequences of martingale differences.

\eqref{4mb} The inequality
$$\tau\Big(\Phi\Big(|\sum_{k\geq0}x_k|\Big)\Big)\lesssim_{\Phi}\max\Big\{\tau\Big(\Phi\Big[\Big(\sum_{k\geq0}|x_k|^2\Big)^{1/2}\Big]\Big),\tau\Big(\Phi\Big[\Big(\sum_{k\geq0}|x_k^*|^2\Big)^{1/2}\Big]\Big)\Big\},$$
is established in \cite[Corollary 3.3]{DR}. The inequality
$$\tau\Big(\Phi\Big(|\sum_{k\geq0}x_k|\Big)\Big)\gtrsim_{\Phi}\max\Big\{\tau\Big(\Phi\Big[\Big(\sum_{k\geq0}|x_k|^2\Big)^{1/2}\Big]\Big),\tau\Big(\Phi\Big[\Big(\sum_{k\geq0}|x_k^*|^2\Big)^{1/2}\Big]\Big)\Big\}.$$
follows from Lemma \ref{right c head} and \eqref{arbitrary to tensor1}.
\end{proof}

We end this paper by giving some examples to illustrate that the Burkholder-Gundy inequalities in Theorem \ref{second main} (respectively, Theorem \ref{fourth main}) can be applied to a wider class of symmetric spaces (respectively, Orlicz functions) than the main results in \cite{DPPS,J-BG} (respectively, \cite{BC}).

\begin{example}\label{example1} Consider the family of functions
$$\Phi(t)=t^p\ln^\alpha(e+t),\quad t\geq0,$$ where $1<p<\infty$ and $\alpha \in \mathbb R$. The function $\Phi$ is a convex function for $\alpha\geq0$ and is equivalent to a convex function for $\alpha<0.$ It is easy to see that $p_{\Phi}=q_{\Phi}=p.$ Moreover, $\Phi$ is $p$-convex for $\alpha\geq0$ and $r$-convex for each $1\leq r<p$ if $\alpha<0$ \cite[page 284]{SS}. If $p=2$, then \cite[Theorem 5.1]{BC} is not applicable, whereas Theorem \ref{fourth main} above still yields $\Phi$-moment Burkholder-Gundy inequalities.
\end{example}

\begin{example}\label{example2} Let $\Phi(t)=t^p\log(1+t^q)$ with $p>1$ and $q>0$. It is easy to check that $\Phi$ is an Orlicz function with $p_{\Phi}=p$ and $q_{\Phi}=p+q.$
\begin{enumerate}[{\rm (i)}]
\item\label{5ma} Suppose that $p=2.$ Then \cite[Theorem 5.1]{BC} gives no information. However, by \cite[Lemma 8]{JSZ} it is not hard to see that $\Phi$ is $2$-convex and $(2+q)$-concave, and hence the corresponding Burkholder-Gundy inequality holds due to Theorem \ref{fourth main} (\ref{4mb}).
\item\label{5mb} Suppose that $p+q=2.$ Then \cite[Theorem 5.1]{BC} also gives no information. However, $\Phi$ is $2$-concave, and hence the corresponding Burkholder-Gundy inequality holds due to Theorem \ref{fourth main} (\ref{4ma}).
\end{enumerate}
\end{example}

\begin{example} Let $\Phi$ be as in Example \ref{example1} and Example \ref{example2} and let $L_\Phi(0,1)$ be the corresponding Orlicz space. Then Theorem \ref{second main} can be applied to the symmetric space $E=L_\Phi(0,1)$, while \cite[Proposition 4.18]{DPPS} and \cite[Theorem 3.1]{J-BG} are both inefficient. In particular, in the setting of Example \ref{example2} with $p=2$ and $p+q\leq4$, Theorem \ref{the 2-4} is applicable to the symmetric space $E=L_\Phi(0,1)$ and yields the corresponding Burkholder inequalities; in fact since $\Phi$ is $2$-convex and $p+q$-concave, $L_\Phi(0,1)$ is $2$-convex and $p+q$-concave (\cite[Theorem 50]{JSZ}), and hence $L_\Phi(0,1)\in Int (L_2(0,1),L_4(0,1))$ (\cite{KM} or \cite[Theorem 4.31]{D1}).

\end{example}

\appendix

\section{Proof of Theorem \ref{junge decomposition}}

The proof below is due to Junge, Proposition 2.8 in \cite{J-D}.

We use the theory of Hilbert modules presented in \cite{Lance} (in particular, Kasparov stabilisation theorem).

Hilbert modules $(H_1,\langle\cdot,\cdot\rangle_1)$ and $(H_2,\langle\cdot,\cdot\rangle_2)$ are said to be isomorphic (see p.24 in \cite{Lance}) if  there exists a linear bijection $w:H_1\to H_2$ such that
$$\langle w(x),w(y)\rangle_2=\langle x,y\rangle_1,\quad x,y\in H_1.$$
Consider the standard Hilbert module $H_{\mathcal{N}}$ as in \cite{J-D},
$$H_{\mathcal{N}}:=\{x\in\mathcal{N}\bar{\otimes}\mathcal{B}(\ell_2):\ x\cdot (1\otimes e_{11})=x\},\quad $$
and
$$x^*y=\langle x,y\rangle_{H_{\mathcal{N}}}\otimes e_{11},\quad x,y\in H_{\mathcal{N}}.$$

\begin{proof}[Proof of Theorem \ref{junge decomposition}] Let $X\subset\mathcal{M}$ be a separable $C^*$-subalgebra which is dense in $L_2(\mathcal{M}).$ Let $F$ be a closure of $X$ with respect to the norm $x\to\|\mathcal{E}(|x|^2)\|_{\N}^{1/2}.$ It follows that $F$ is a countably generated Hilbert module over $\mathcal{N}$ when equipped with the $\mathcal{N}$-valued inner product
$$(x,y)\to\mathcal{E}(x^*y).$$
By \cite[Theorem 6.2]{Lance}, Hilbert modules $F\oplus H_{\mathcal{N}}$ and $H_{\mathcal{N}}$ are isomorphic. Hence, there exists a mapping $w:F\oplus H_{\mathcal{N}}\to H_{\mathcal{N}}$ such that
$$\langle w(x_1,x_2),w(y_1,y_2)\rangle_{H_{\mathcal{N}}}=\mathcal{E}(x_1^*y_1)+\langle x_2,y_2\rangle_{H_{\mathcal{N}}}.$$
Setting $x_2=y_2=0,$ we obtain
$$\langle w(x_1,0),w(y_1,0)\rangle_{H_{\mathcal{N}}}=\mathcal{E}(x_1^*y_1).$$
Therefore,
$$w(x_1,0)^*w(y_1,0)=\mathcal{E}(x_1^*y_1)\otimes e_{11}.$$
Consider the mapping $u:F\to H_{\mathcal{N}}$ defined by the setting $u(x)=w(x,0).$ In particular, $u:X\to\mathcal{N}\bar{\otimes}\mathcal{B}(\ell_2)$ and
$$u(x)^*u(y)=\mathcal{E}(x^*y)\otimes e_{11},\quad x,y\in X.$$
Clearly,
$$\|u(x)\|_{L_2(\mathcal{N}\bar{\otimes}\mathcal{B}(\ell_2))}=\|(\mathcal E(|x|^2))^{1/2}\|_{L_2(\N)}=\|x\|_{L_2(\N)}.$$
Hence, $u:L_2(\mathcal{M})\to L_2(\mathcal{N}\bar{\otimes}\mathcal{B}(\ell_2))$ is an isometry.
\end{proof}

\noindent{\bf Acknowledgements}
This work was completed when the first named author was visiting UNSW. He would like to express his gratitude to the School of Mathematics and Statistics of UNSW for its warm hospitality.


\begin{thebibliography}{90}
\bibitem{AK} Albiac F.,  Kalton N. {\it  Topics in Banach space theory.} Graduate Texts in Mathematics, 233. Springer, New York, (2006).
\bibitem{AS2014} Astashkin S., Sukochev F. {\it Orlicz sequence spaces spanned by identically distributed independent random variables in $L_p$-spaces.} J. Math. Anal. Appl.  {\bf413}  (2014),  no. 1, 1--19.
\bibitem{AS2005} Astashkin S., Sukochev F. {\it Series of independent random variables in rearrangement invariant spaces: an operator approach.} Israel J. Math. {\bf 145} (2005), 125--156.
\bibitem{ASS} Astashkin S., Semenov E., Sukochev F. {\it The Banach-Saks p -property.} Math. Ann.  {\bf332}  (2005),  no. 4, 879--900.
\bibitem{ASW2} Astashkin S., Sukochev F., Wong C. {\it Distributionally concave symmetric spaces and uniqueness of symmetric structure.} Adv. Math.  {\bf232}  (2013), 399--431.
\bibitem{ASW} Astashkin S., Sukochev F., Wong C. {\it Disjointification of martingale differences and conditionally independent random variables with some applications.} Studia Math. {\bf 205} (2011), no. 2, 171--200.
\bibitem{pacific}  Astashkin S., Sukochev F., Zanin D. {\it Disjointification inequalities in symmetric quasi-Banach spaces and their applications.} Pacific J. Math. {\bf 270} (2014), no. 2, 257--285.
\bibitem{Bekjan-rocky} Bekjan T. {\it $\Phi$-inequalities of noncommutative martingales.} Rocky Mountain J. Math. {\bf 36} (2006), no. 2, 401--412.
\bibitem{BC} Bekjan T., Chen Z. {\it Interpolation and $\Phi$-moment inequalities of noncommutative martingales.} Probab. Theory Related Fields {\bf 152} (2012), no. 1-2, 179--206.
\bibitem{BCO} Bekjan T., Chen Z., Osekowski A. {\it Noncommutative maximal inequalities associated with convex functions.} to appear Trans. Amer. Math. Soc.(2015).
\bibitem{BergLofstrom} Bergh J., L\"ofstr\"om J. {\it Interpolation spaces. An introduction.} Grundlehren der Mathematischen Wissenschaften, No. 223. Springer-Verlag, Berlin-New York, 1976.
\bibitem{BDG} Burkholder D., Davis B., Gundy, R. {\it Integral inequalities for convex functions of operators on martingales.}  Proceedings of the Sixth Berkeley Symposium on Mathematical Statistics and Probability (Univ. California, Berkeley, Calif., 1970/1971), Vol. II: Probability theory,  pp. 223--240. Univ. California Press, Berkeley, Calif., 1972.
\bibitem{BG} Burkholder D., Gundy R. {\it Extrapolation and interpolation of quasi-linear operators on martingales.} Acta Math. {\bf 124}  (1970), 249--304.
\bibitem{CD1} Carothers N., Dilworth S. {\it Inequalities for sums of independent random variables.} Proc. Amer. Math. Soc. {\bf 104} (1988), no. 1, 221--226.
\bibitem{D1} Dirksen S. {\it Noncommutative and vector-valued Rosenthal inequalities.} Thesis, Delft University of Technology (2011).
\bibitem{dirksen2015} Dirksen S.{\it Noncommutative Boyd interpolation theorems.} Trans. Amer. Math. Soc.(2015).
\bibitem{DPPS} Dirksen S., de Pagter B., Potapov D., Sukochev F. {\it Rosenthal inequalities in noncommutative symmetric spaces.} J. Funct. Anal. {\bf 261} (2011), no. 10, 2890--2925.
\bibitem{DR} Dirksen S., Ricard E. {\it Some remarks on noncommutative Khintchine inequalities.} Bull. Lond. Math. Soc.  {\bf 45}  (2013),  no. 3, 618--624.
\bibitem{DDP-Int} Dodds P.,  Dodds T., de Pagter B.  {\it  Fully symmetric operator spaces.} Integral Equations Operator Theory {\bf 15} (1992), no. 6, 942--972.
\bibitem{DSZ} Dykema K., Sukochev F., Zanin D. {\it A decomposition theorem in II$_1$--factors.} J. Reine Angew. Math. (to appear).
\bibitem{FK} Fack T., Kosaki H. {\it Generalized $s$-numbers of $\tau$-measurable operators.} Pacific J. Math. {\bf 123} (1986), no. 2, 269--300.
\bibitem{G} Garsia  A. {\it On a convex function inequality for martingales.} Ann. Probability {\bf 1}  (1973), no. 1, 171--174.
\bibitem{J-B} Jiao Y. {\it Burkholder inequalities in noncommutative Lorentz spaces.} Proc. Amer. Math. Soc. {\bf 138} (2010), no. 7, 2431--2441.
\bibitem{J-BG} Jiao Y. {\it  Martingale inequalities in noncommutative symmetric spaces.} Arch. Math. {\bf 98} (2012), no. 1, 87--97.
\bibitem{JSXZ} Jiao Y., Sukochev F., Xie G., Zanin D. {\it $\Phi$-moment inequalities for independent and free independent random variables.} J. Funct. Anal. {\bf 270} (2016), pp. 4558--4596.
\bibitem{JSZ-LMS} Jiao Y., Sukochev F., Zanin D.{\it Johnson-Schechtman and Khintchine inequalities in noncommutative probability theory.} J. Lond. Math. Soc. {\bf(2) 94} (2016), no. 1, 113--140.
\bibitem{JS martingale} Johnson W., Schechtman G. {\it Martingale inequalities in rearrangement invariant function spaces.} Israel J. Math.  {\bf 64}  (1988),  no. 3, 267--275.
\bibitem{JS} Johnson W., Schechtman G. {\it Sums of independent random variables in rearrangement invariant function spaces.} Ann. Probab. {\bf 17} (1989), no. 2, 789--808.
\bibitem{J-D} Junge M. {\it Doob's inequality for non-commutative martingales.} J. Reine Angew. Math. {\bf 549} (2002), 149--190.
\bibitem{JPX} Junge M., Parcet J., Xu Q. {\it Rosenthal type inequalities for free chaos.} Ann. Probab. {\bf 35} (2007), no. 4, 1374--1437.
\bibitem{JSZ} Junge M., Sukochev F., Zanin D. {\it Embedding of operator ideals into $L_p$-spaces on finite von Neumann algebra.} submmitted manuscript. (2015).
\bibitem{JX1} Junge M., Xu, Q. {\it Noncommutative Burkholder/Rosenthal inequalities.} Ann. Probab.  {\bf 31}  (2003),  no. 2, 948--995.
\bibitem{KM} Kalton N.,  Montgomery-Smith S. {\it  Interpolation of Banach spaces.}  Handbook of the geometry of Banach spaces, Vol. 2,  1131--1175, North-Holland, Amsterdam, (2003).
\bibitem{KS} Kalton N., Sukochev F., {\it Symmetric norms and spaces of operators.} J. Reine Angew. Math. {\bf 621}  (2008), 81--121.
\bibitem{KPS} Krein  S., Petunin  Y., Semenov E. {\it Interpolation of linear operators}, Translations of Mathematical Monographs, Amer. Math. Soc. {\bf 54} (1982).
\bibitem{Lance} Lance E. {\it Hilbert $C^*$-modules. A toolkit for operator algebraists.} London Mathematical Society Lecture Note Series. {\bf 210}. Cambridge: Univ. Press, ix, 130 p.
\bibitem{LT} Lindenstrauss J., Tzafriri L. {\it Classical Banach spaces. II. Function spaces.} Ergebnisse der Mathematik und ihrer Grenzgebiete, {\bf 97}. Springer-Verlag, Berlin-New York, 1979.
\bibitem{LSZ} Lord S., Sukochev F., Zanin D. {\it Singular traces. Theory and applications.} De Gruyter Studies in Mathematics, {\bf 46}. De Gruyter, Berlin, 2013.
\bibitem{LSh} Lorentz G., Shimogaki T. {\it Interpolation theorems for the pairs of spaces $(L_p,L_{\infty})$ and $(L_1,L_q).$} Trans. Amer. Math. Soc. {\bf 159} (1971), 207--221.
\bibitem{LPP} Lust-Piquard F., Pisier G. {\it Noncommutative Khintchine and Paley inequalities.} Arkiv for Mat. {\bf 29} (1991): 241--260.
\bibitem{Mali thesis} Maligranda L. {\it Indices and interpolation.} Dissert. Math., {\bf234.} Polska Akademia Nauk, Inst. Mat., 1985.
\bibitem{M2013} Maligranda L. {\it The $K$-functional for $p$-convexifications.} Positivity. {\bf 17} (2013): 707--710.
\bibitem{lM-Suk} Le Merdy C., Sukochev F. {\it Rademacher averages on noncommutative symmetric spaces.} J. Funct. Anal. {\bf 255} (2008), no. 12, 3329--3355.
\bibitem{PX} Pisier G., Xu Q. {\it Non-commutative martingale inequalities.} Comm. Math. Phys. {\bf 189} (1997), no. 3, 667--698.
\bibitem{MM} Rao M., Ren Z. {\it Applications of Orlicz spaces.} Monographs and Textbooks in Pure and Applied Mathematics, {\bf 250}. Marcel Dekker, Inc., New York, 2002.
\bibitem{N} Randrianantoanina N. {\it Non-commutative martingale transforms.} J. Funct. Anal.  {\bf194}  (2002),  no. 1, 181--212.
\bibitem{N2} Randrianantoanina N. {\it Conditioned square functions for noncommutative martingales.} Ann. Probab. {\bf 35} (2007), no. 3, 1039--1070.
\bibitem{NW} Randrianantoanina N., Wu L. {\it Martingale inequalities in noncommutative symmetric spaces.} J. Funct. Anal. (2015), http://dx.doi.org/10.1016/j.jfa.2015.05.017.
\bibitem{R} Rosenthal H. {\it On the subspaces of $L_p$ ($p>2$) spanned by sequences of independent random variables.} Israel J. Math. {\bf 8} (1970), 273--303.
\bibitem{SS} Semenov E., Sukochev F. {\it The Banach-Saks index.} Sb. Math. {\bf195} (2004), no. 1-2, 263--285
\bibitem{S} Sukochev F. {\it Completeness of quasi-normed symmetric operator spaces.} Indag. Math. (N.S.) {\bf25}  (2014),  no. 2, 376-388.
\bibitem{SZ} Sukochev F., Zanin D. {\it Johnson-Schechtman inequalities in the free probability theory.} J. Funct. Anal. {\bf 263} (2012), no. 10, 2921--2948.
\bibitem{V1} Voiculescu D. {\it A strengthened asymptotic freeness result for random matrices with applications to free entropy.} Internat. Math. Res. Notices 1998, no. {\bf 1}, 41--63.


\end{thebibliography}
\end{document}